\documentclass[UTF-8,reqno]{amsart}
\usepackage{enumerate}
\usepackage{mhequ}
\usepackage{amssymb,url,color, booktabs,nccmath}
\usepackage[left=3.2cm,right=3.2cm,top=3cm,bottom=3cm]{geometry}
\usepackage{mathrsfs}
\usepackage{enumitem,dsfont}
\usepackage{amsmath,amsfonts, verbatim, amsthm,amssymb,bbm,mathtools,mathrsfs,authblk, esint}
\usepackage{bm}
\usepackage{extarrows}

\usepackage{color}
\usepackage[colorlinks=true]{hyperref}
\hypersetup{
	linkcolor=blue,          
	citecolor=red,        
	filecolor=blue,      
	urlcolor=cyan
}

\definecolor{darkergreen}{rgb}{0.0, 0.5, 0.0}

\usepackage{fancyhdr} 
\numberwithin{equation}{section}
\def\theequation{\arabic{section}.\arabic{equation}}
\newcommand{\be}{\begin{eqnarray}}
	\newcommand{\ee}{\end{eqnarray}}
\newcommand{\ce}{\begin{eqnarray*}}
	\newcommand{\de}{\end{eqnarray*}}
\newtheorem{theorem}{Theorem}[section]
\newtheorem{lem}[theorem]{Lemma}
\newtheorem{prop}[theorem]{Proposition}

\newtheorem{cor}[theorem]{Corollary}

\newtheorem{definition}[theorem]{Definition}
\theoremstyle{definition}
\newtheorem{remark}[theorem]{Remark}

\newcommand{\RR}{\mathbb{R}} 
\newcommand{\fH}{\mathcal{H}} %
\newcommand{\EE}{\mathbb{E}} 
\newcommand{\PP}{\mathbb{P}} 
\newcommand{\fF}{\mathcal{F}} 
\newcommand{\1}{\mathbbm{1}} 
\newcommand{\fL}{\mathcal{L}} 

\newcommand{\Lp}{\pmb{\Pi}} 
\newcommand{\TT}{\mathbb{T}} 
\newcommand{\fS}{\mathcal{S}} 
\newcommand{\ZZ}{\mathbb{Z}} 
\newcommand{\HH}{\mathbb{H}} 
\newcommand{\FK}{\Gamma L^2}
\newcommand{\fG}{\mathcal{G}}
\newcommand{\fA}{\mathcal{A}}
\newcommand{\NN}{\mathbb{N}}
\newcommand{\fN}{\mathcal{N}}
\newcommand{\bP}{\mathbf{P}}
\newcommand{\bE}{\mathbf{E}}
\newcommand{\bj}{\mathbf{j}}
\newcommand{\bt}{\mathbf{t}}

\newcommand{\fR}{\mathcal{R}}
\newcommand{\fB}{\mathcal{B}}

\newcommand{\fT}{\mathcal{T}}

\newcommand{\ud}{\operatorname{d}\! }
\newcommand{\authorfootnotes}{\renewcommand\thefootnote{\@fnsymbol\c@footnote}}%
\begin{document}
\title{Gaussian Fluctuations for the Stochastic Landau-Lifshitz Navier-Stokes Equation in Dimension $D\geq2$}
\maketitle
\begin{center}
	Sotiris Kotitsas\footnote{Institute of Mathematics, EPFL, Switzerland. E-mail: sotirios.kotitsas@epfl.ch},
	Marco Romito\footnote{Dipartimento di Matematica, Università di Pisa, Largo Bruno Pontecorvo 5,
	I–56127 Pisa, Italia. E-mail: marco.romito@unipi.it},
	Zhilin Yang\footnote{Academy of Mathematics and Systems Science, Chinese Academy of Sciences, Beijing 100190, China. E-mail: yangzhilin0112@163.com}, 
	Xiangchan Zhu\footnote{Academy of Mathematics and Systems Science, Chinese Academy of Sciences, Beijing 100190, China. E-mail: zhuxiangchan@126.com}
	
\end{center}

\begin{abstract}
We revisit the large-scale Gaussian fluctuations for the stochastic Landau-Lifshitz Navier-Stokes equation (LLNS) at and above criticality, using the method in \cite{CGT24}. With the classical diffusive scaling in $d\geq 3$ and weak coupling scaling in $d=2$,
we obtain the convergence of the regularised LLNS to a stochastic heat equation with a non-trivially renormalized coefficient.
Moreover, we obtain an asymptotic expansion of the effective coefficient when $d\geq3$, and show that the one in \cite[Conjecture 6.5]{JP24} is incorrect.
The new ingredient in our proof is
a case-by-case analysis to track the evolution of the vector under the action of the Leray projection, combined with the use of the anti-symmetric part of the generator and a rotational change of coordinates to derive the desired decoupled stochastic heat equation from the original coupled system.

\end{abstract}

\tableofcontents
\section{Introduction}
In this paper, we focus on the stochastic Landau-Lifshitz Navier-Stokes equation in dimension $d \geq 2$ on the torus $\TT^d=(\RR/ \ZZ)^d$, which is formally given by 
\begin{equation}\label{equ.SNS}
    \partial_t u = \Delta u - \nabla p -\lambda \operatorname{div}(u\otimes u) + \sqrt{2}(-\Delta)^{\frac{1}{2}}\xi\,,\quad \nabla \cdot u = 0\,,
\end{equation}
where $p$ is the pressure which ensures that dynamics preserves the
incompressibility condition, and $\xi$ is a $d$-dimensional space-time white noise with covariance
\[
    \EE[\xi_i(\varphi)\xi_j(\psi)] = \delta_{ij}\langle \varphi,\psi\rangle_{L^2(\RR_+\times\TT^d)}\,,\quad \varphi,\psi \in L^2(\RR_+\times\TT^d)\,.
\]
\cite{JP24} has proved that the solution to \eqref{equ.SNS} has the same distribution as the solution to the following fluctuating hydrodynamics equation of Landau and Lifshitz \cite{LL87}, which was introduced to describe thermodynamic fluctuations in fluids:
\begin{align*}
	 \partial_t u = \Delta u - \nabla p -\lambda \operatorname{div}(u\otimes u) + \sqrt2\nabla\cdot\tau\,,\quad \nabla \cdot u = 0\,,
\end{align*}
where $\tau$ is the centered Gaussian noise with corvariance given by
\[
\EE[\tau_{ij}(\varphi)\tau_{kl}(\psi)] = \left(\delta_{ik}\delta_{jl}+\delta_{il}\delta_{jk}-\frac23\delta_{ij}\delta_{kl}\right)\langle \varphi,\psi\rangle_{L^2(\RR_+\times\TT^d)}\,,\quad \varphi,\psi \in L^2(\RR_+\times\TT^d)\,.
\]
Compared with the deterministic version of the incompressible Navier-Stokes equation, this model, as mentioned in \cite{BGME22} given as the fluctuation-dissipation relation, is more appropriate to describe the dissipation range of turbulence in molecular fluids.

The original equation \eqref{equ.SNS} is ill-posed, since the noise is too irregular for the non-linearity to be well-defined. For fixed $N\in\NN$, we consider a truncated version on $\RR_+ \times N\TT^d$:
\begin{equation}\label{eq:truncated-SNS-intro}
    \partial_t u = \Delta u -\lambda \rho * \Lp \operatorname{div}((\rho*u)\otimes (\rho*u)) + \sqrt{2}(-\Delta)^{\frac{1}{2}}\Lp\xi\,,\quad \nabla \cdot u = 0\,,
\end{equation}
with the mollifier $\rho$ given by $\fF \rho = \1_{B(0,1)}$.
Here $\Lp$ is the Leray projection defined by \eqref{2.1} below.
By \cite[Theorem 3.4]{JP24} there exists a unique strong solution to \eqref{eq:truncated-SNS-intro}, with the invariant measure given by a divergence-free and mean-free space white noise on $N\TT^d$. Then by the diffusive scaling
\[
    u^N(t,x) = N^{\frac{d}{2}}u(N^{2}t,Nx)\,,
\]
we can obtain the following equation on $\RR_+ \times \TT^d$:
\begin{equation}\label{eq:scaled-NS}
    \partial_t u^N = \Delta u^N -\lambda_N \rho^N * \Lp \operatorname{div}((\rho^N*u^N)\otimes (\rho^N*u^N)) + \sqrt{2}(-\Delta)^{\frac{1}{2}}\Lp\xi\,,\quad \nabla \cdot u^N = 0\,,
\end{equation}
where $\rho^N(x) = N^d\rho(Nx)$ and $\lambda_N=\lambda N^{1 -\frac{d}{2}}$. Therefore, \eqref{equ.SNS} in $d=2$ and $d\geq3$ belong to critical and supercritical regimes respectively in the sense of regularity structures \cite{Hai14}, for which the theory of regularity structures \cite{Hai14}, paracontrolled calculus \cite{GIP15}, or the flow approach \cite{Duc22} are not applicable. Also, it falls out of the scope of the energy solution approach \cite{GPP24} to define a stationary energy solution. In fact, \cite{JP24} has considered the large-scale behavior of \eqref{equ.SNS} and made a conjecture for $d\geq3$. 
In this paper, we revisit the diffusive scaling limit in $d\geq3$ with $\lambda_N=\lambda N^{1 -\frac{d}{2}}$ and the
weak coupling limit in $d=2$ with
$\lambda_N=\frac{\lambda}{\sqrt{\log N}}.$ 
{Let $\fS(\TT^d,\RR^d)$ be the space of infinitely differentiable $\RR^d$-valued functions on $\TT^d$ and let $\fS'(\TT^d,\RR^d)$ be its dual space.}
Now we state the main result.

\begin{theorem}\label{thm:main}
	Let $T>0$ and $\lambda$ be a fixed positive constant. For $N>0$, let
	\begin{equation}\label{lambda}
		\lambda_N=\left\{
		\begin{aligned}
			\frac{\lambda}{\sqrt{\log N}} & , &d=2,\\
			\lambda N^{1 - \frac{d}{2}} & , &d\geq3.
		\end{aligned}
		\right.
	\end{equation}
	Define $u^N$ be the stationary solution to \eqref{eq:scaled-NS} with the initial value $u^N(0,\cdot):=\mu,$ for $\mu$
	a divergence-free and mean-free spatial white noise on $\TT^d.$
	Then there exists a strictly positive constant $D$, depending only on the dimension $d$ and $\lambda$, such that $u^N$ converges in law in $C([0,T],\fS'(\TT^d,\RR^d))$ to the unique stationary solution $u$ to
	\begin{align}\label{eq:limeq}
	\partial_t u = (1+D)\Delta u + \sqrt{2(1+D)}(-\Delta)^{\frac{1}{2}}\Lp\xi\,,\quad u(0,\cdot)=\mu\,.
	\end{align}
	In the case $d=2$, $D=\sqrt{\frac{\lambda^2}{8\pi}+1}-1.$
	{In the case $d\geq 3,$ the following asymptotic expansion of $D$ holds for any $m\geq1$: 
	$$D=\sum_{l=1}^m f_l\lambda^{2l} + R_m \,,$$
	with $f_l\in\RR$ independent of $\lambda$ and $R_m\in\RR$ given in \eqref{eq:expansion}. Moreover, there exists a positive constant $C$, depending only on $d$, such that 
	$$|f_l|\leq l!C^{l},\quad |R_m|\leq (m+1)!C^{m+1}\lambda^{2m+2}.$$}
\end{theorem}

\begin{remark}
	Theorem \ref{thm:main} establishes the weak convergence of $u^N$ with $u^N(0,\cdot)=\mu$. In par
	ticular, the initial data is random. 
	The large-scale behavior of $u^N$ when started by some fixed initial data $u_0$ is more subtle and out of the scope of this paper. Instead, by adapting
	the methods used here, one can show a "semi-quenched" convergence. Specifically, we denote by $\bE_{u_0}^N$ and $\bE_{u_0}$ the expectations with respect to the laws of $u^N$ and $u$ starting from the fixed initial value $u_0$, and denote by $P^N_t$ and $P_t$ their
	Markov semigroups respectively. 
	Let $\mu$ be the law of the divergence-free and mean-free spatial white noise.
	Then, for every $p,k\geq1$, every $F_1\cdots,F_k\in L^p(\mu)$ and
	$0 \leq t_1 \leq\cdots\leq t_k$, we could obtain the following by the resolvent convergence,
	\begin{align}
		\sup_{t\geq0}\|P_t^NF_1-P_tF_1\|_{L^p(\mu)}&\to0,\label{1.6}\\
		\bE_{u_0}^N[F_1(u^N(t_1))\cdots F_k(u^N(t_k))]&\to \bE_{u_0}[F_1(u(t_1))\cdots F_k(u(t_k))],\label{1.7}
	\end{align}
	as $N\to\infty$, where the convergence is in $L^p(\mu)$.
	This type of convergence has appeared recently in \cite{CMT24} for the critical case, which is not
	quenched convergence (i.e. convergence a.s. with respect to the initial data), but stronger than the annealed convergence of finite-dimensional
	distributions (i.e. convergence with $u_0\sim\mu$).
	
	We briefly sketch the proof of \eqref{1.6} and \eqref{1.7}. By the Trotter-Kato theorem \cite[Theorem 4.2]{IK02} and \cite[Theorem 3,53]{Gra25}, it suffices to prove a generalization version of Proposition \ref{prop:limite}, where $\tilde\sigma_{\bj,\bt}$ (resp. $\tilde\sigma_{-\bj',\bt'}$) is replaced with $\sigma_{j_1,t_1}\otimes\cdots\otimes\sigma_{j_r,t_r}$ (resp. $\sigma_{-j'_1,t'_1}\otimes\cdots\otimes\sigma_{-j'_r,t'_r}$) for every $r\geq1$. 
	This can be achieved by extending Lemma \ref{lem5.8} and Lemma \ref{lem5.9} to general $r\geq1$, and the proofs proceed in the same manner.
	
\end{remark}

\begin{remark}\label{rmk}
	In the case $d\geq3$, by the same arguments as in Proposition \ref{prop:A} in the two-dimensional case, we can derive a similar "Replacement Lemma": for any $c_1>0,\, n\in\NN$ and the symmetric divergence-free test functions $\psi_1,\psi_2$ in $L^2(\TT^{dn},\RR^{dn})$,
	\begin{align}\label{4.1}
		\left|\langle(-\fL_0+c_1(-\fL_0))^{-1}\fA^N_+\psi_1,\fA^N_+\psi_2\rangle-\frac{c\lambda^2}{1+c_1}\langle(-\fL_0)\psi_1,\psi_2\rangle\right|
		\lesssim&\|\fN(-\fL_0)^{\frac12}\psi_1\|\|\fN(-\fL_0)^{\frac12}\psi_2\|,
	\end{align}
	where $c=\lambda^{-2}\lim_{N\to\infty}\|(-\fL_0)^{-\frac12}\fA_+^N(-\fL_0)^{-\frac12}\sigma_{k,1}\|^2,$ $\sigma_{k,1}$ is given by \eqref{vector-fields},
	$\fL^N = \fL_0 +\fA_+^N +\fA_-^N$ is the generator of $u^N$, and $\fN$ denotes the number operator (see Theorem \ref{thm.Geneator_in_Fourier} and Definition \ref{def:num}).
	Let $D_{\mathrm{rep}}$ be the unique positive solution to $x = \frac{c\lambda^2}{1+x}$. Although there is no vanishing term on the right hand side of \eqref{4.1}, 
	one might be tempted to regard the symmetric operator $D_{\mathrm{rep}}(-\fL_0)$ as a "Replacement operator" for approximating the fixed-point 
	$\fH^N = \fA^N_-(-\fL_0+\fH^N)^{-1}\fA_+^N$, and conjecture that $D=D_{\mathrm{rep}}$ as in the two dimensional case. The following corollary shows that such conjecture is in fact incorrect.
\end{remark}
\begin{cor}\label{cor}
	For $d=3$, the positive constant $D$ in Theorem \ref{thm:main} does not coincide with either $\nu_{\mathrm{eff}}-1$ proposed in \cite[Conjecture~6.5]{JP24} or $D_{\mathrm{rep}}$ defined in Remark \ref{rmk}.
\end{cor}
\begin{remark}
	For $d>3$, we believe that the same conclusion as in Corollary~\ref{cor} should also hold, which can be verified by following the proof in Section~\ref{sec5} and evaluating the integrals in \eqref{5.7} and \eqref{5.8}.
\end{remark}

The proof strategy of Theorem \ref{thm:main} follows the framework proposed in \cite{CGT24}. The key step is the derivation of a Fluctuation-Dissipation relation (Theorem \ref{thm:F-D}), which approximates the non-linearity term $-\lambda_N\int_0^tB^N(u_s^N)(\varphi)\ud s$ in \eqref{3.1} by a drift $\int_0^tD\fL_0^Nu_s^N(\varphi)\ud s$, capturing the additional diffusivity, along with a Dynkin martingale that represents the extra noise.
By truncating the generator $\fL^N$ of $u^N$, the martingale term can be obtained by solving the truncated resolvent equation $-\fL^N_{\geq 2}v^N = \fA^N_+\varphi\,.$
For the critical case $d=2$, we can approximate the resolvent solution by a sequence of diagonal symmetric operators through a Replacement Lemma. This approximation approach originates from \cite{CET23b}, and was simplified in \cite{CG24,CGT24}. In the supercritical case $d\geq3,$ we follow the routine in \cite{CGT24} to analyze $v^N.$ Compared with the stochastic Burgers equation, the solution $u^N$ to \eqref{eq:scaled-NS} is vector-valued and divergence-free. 
{The presence of the Leray projection in the nonlinearity makes all vector components of $u^N$ be coupled, which destroys the componentwise structure one has in the scalar case.
To address this issue, we
perform a case-by-case analysis depending on the action of the Leray projection, and use the anti-symmetric part of the generator, along with a rotational change of coordinates in the integral, to derive the desired decoupled limiting stochastic heat equation from the original coupled system \eqref{eq:scaled-NS}.}
We also give an asymptotic expansion for $D$ in $d\geq3$ by induction argument and variational formula.
In this case, both the diagonal and off terms of $\langle(-\fL_0)^{-1}\fA_+^N\psi,\fA_+^N\psi\rangle$ for $\psi\in\FK$ could produce extra terms as $N\to\infty$, which cannot be approximated by a replacement diagonal operator.

{The rest of this paper is organized as follows. In Section \ref{sec2}, we recall some well-known results on Wiener chaos decomposition and give the basic properties of \eqref{eq:truncated-SNS-intro} and preliminary
estimates. In section \ref{sec3} we
reduce the main statement to the proof of the so-called Fluctuation-Dissipation
Theorem.
The bulk of the paper is Section \ref{sec4}, in which we prove Theorem \ref{thm:main} in $d\geq3$. 
Necessary results and proofs analogous to those in \cite{CGT24} are included in Appendix \ref{appendixB}.
In section \ref{sec5}, we establish an asymptotic expansion for $D$ and prove Corollary \ref{cor}.
Finally, the case $d=2$ for Theorem \ref{thm:main} is addressed in Section \ref{sec6}, with further technical estimates given in Appendix \ref{appenA1}.}

\section{Notations and preliminaries}\label{sec2}
Throughout the paper, we employ the notation $a\lesssim b$ if there exists a constant $c>0$
such that $a \leq cb$, and $\lesssim_d$ means the constant $c = c(d)$ depending on $d$.
{For $\varphi\in\fS(\TT^d,\RR^d)$,
we denote its $l$-th component by $\varphi_l$, for $l=1,\cdots,d.$
For $k\in \ZZ^d$ and $x\in\TT^d$, set $e_k(x) = e^{2\pi\iota k\cdot x}$.
The Fourier transform of $\varphi$ at $k\in\ZZ^d$ is denoted by $\fF(\varphi)(k)$ or $\hat\varphi(k)$, whose $l$-th component is defined by
\[
\fF(\varphi)(l,k) = \hat{\varphi}(l,k) := \int_{\TT^d}\varphi_l(x)e_{-k}(x)\ud x \,.
\]
For $\eta\in\fS'(\TT^d,\RR^d)$, we denote
its Fourier transform by $\hat\eta(k):=\eta(e_{-k})$.}
We introduce the Leray projection $\Lp$, which maps $\varphi \in \fS(\TT^d,\RR^d)$ to a divergence free vector. Denote the Leray projection matrix as $\hat\Lp(k)=\mathbf{I}-\frac{1}{|k|^2}k\otimes k$ {for every $k\in\ZZ_0^d:=\ZZ^d\backslash\{0\}.$} Then
\begin{align}\label{2.1}
	\widehat{\Lp\varphi}(k) = \hat{\Lp}(k)\hat{\varphi}(k)\,.
\end{align}
{Let $\HH$ and $\HH_{\mathbb C}$ be the mean-zero and divergence-free subspace of $L^2(\TT^d,\RR^d)$ and $L^2(\TT^d,\mathbb C^d)$ respectively.
For $\varphi\in \HH^{\otimes n}$, we write $\varphi(l_{1:n},x_{1:n})$ to denote the $(l_1,\cdots,l_n)$-component of $\varphi$ evaluated at the spatial points $x_1,\cdots,x_n\in\TT^d,$ where each $l_i\in\{1,\cdots,d\}$
specifies the component index of the $i$-th factor.}
{For $n\in\NN$, set $L^2_n:=\HH^{\otimes n}$ with the inner product 
	\begin{align}\label{L^2_n}
		\langle f,g\rangle_{L^2_n}:=\langle f,g\rangle_{L^2(\TT^d,\RR^d)^{\otimes n}}\,,\quad f,g\in L^2_n.
	\end{align}}

\subsection{The divergence-free vector field space}\label{sec:2.1}
We recall the following basis of $\HH_{\mathbb C}$. Let $\ZZ^d_0= \ZZ^d_+ \cup \ZZ^d_-$ be a partition of $\ZZ^d_0$ such that
$\ZZ^d_+ \cap \ZZ^d_-= \emptyset$ and $\ZZ^d_+ = - \ZZ^d_-.$
For any $k\in \ZZ^d_+$, let $\{a_{k,1},\cdots, a_{k,d-1}\}$ be an orthonormal basis of $k^\perp := \{x\in \RR^d: k\cdot x=0\}$ such that $\{a_{k,1}, \cdots, a_{k,d-1}, \frac{k}{|k|}\}$ is right-handed. 
For $\alpha=1,\dots,d-1$, we assume $a_{k,\alpha}=a_{|k|^{-1}k,\alpha}$ for $k\in \ZZ^d_+$, and define $a_{k,\alpha} = a_{-k,\alpha}$ for $k\in \ZZ^d_-$.
For every $k\in \ZZ^d_0$ and $\alpha=1,\dots,d-1$, we define the divergence free vector field
    \begin{equation}\label{vector-fields}
    	\sigma_{k,\alpha}(x) := a_{k,\alpha} e_k(x)\,, \quad x\in \TT^d \,.
    \end{equation}
    Then $\{\sigma_{k,1},\cdots,\sigma_{k,d-1}:  k\in \ZZ^d_0\}$ is a CONS of $\HH_{\mathbb C}$. 
    We denote $a_{k,\alpha}^l$ by the $l-$th component of the vector $a_{k,\alpha}.$
    For $\varphi \in \HH$, we have
    \begin{align}
    	\varphi(x) = \sum_{\alpha=1}^{d-1}\sum_{k\in\ZZ_0^d}\varphi_{k,\alpha}\sigma_{k,\alpha}(x)\,,\label{eq:div}
    \end{align}
    where 
    \begin{align}\label{eq:div-free-basis}
    	\varphi_{k,\alpha}=\langle\varphi,\sigma_{k,\alpha}\rangle_{L^2(\TT^d,{\mathbb{C}}^d)} =\sum_{l=1}^{d} a_{k,\alpha}^l\hat\varphi(l,k)\,.
    \end{align}

    \subsection{Wiener space analysis} \label{sec:2.2}
    Let $(\Omega,\fF,\PP)$ be a complete probability space and $\eta$ be a mean-zero and divergence-free spatial white noise, i.e. $\eta$ is a centered Gaussian field with covariance
    \[
        \EE[\eta(f_1)\eta(f_2)] = \left\langle f_1,f_2\right\rangle_{L^2(\TT^d,\RR^d)}\,,\quad f_1,f_2 \in \HH\,.
    \]
    We denote the distribution of $\eta$ as $\mu$. Since $\eta$ is an isonormal Gaussian process, by \cite[Theorem 1.1.1]{Nua06}, the space $L^2(\mu)$
    admits an orthogonal decomposition in terms of the homogeneous Wiener chaoses. The $n-$th Wiener chaos is defined via 
       \begin{align*}
    	\mathscr{H}_{n} :=& \overline{\text{span}\{H_n(\eta(h)): h \in \HH, \,\|h\|_{L^2(\TT^d,\RR^d)} = 1\}}\,,
    \end{align*}
    where $H_n$ are the Hermite polynomials.
    There exists an isomorphism $I$ between the Fock space $\FK=\overline{\oplus_{n=0}^\infty\FK_n}$ and $L^2(\mu)$,
    where $\FK_n$ is the symmetric subspace of $\HH^{\otimes n}$, i.e. for $f\in\FK_n$ and any permutation $\sigma$, $f(l_{1:n},x_{1:n})=f(l_{\sigma(1:n)},x_{\sigma(1:n)})$.
    	{The restriction $I_n$ of $I$ to $\FK_n$
    	is itself an isomorphism from $\FK_n$ to $\fH_n$,} and by \cite[Therorem 1.1.2]{Nua06},
    for any $F \in L^2(\mu)$, there exists a family of kernels $(f_n)_n\in\FK$ such that 
    $F = \sum_{n=0}^\infty I_n(f_n),$ and 
    $$\EE[F^2]=\sum_{n=0}^\infty n!\|f_n\|_{L^2(\TT^d,\RR^d)^{\otimes n}}^2.$$ 
    We take the right hand side as the definition of the scalar product on $\FK$, i.e.
    \begin{align}\label{FK}
    	\langle f,g\rangle=\sum_{n=0}^\infty \langle f_n,g_n\rangle:=\sum_{n=0}^\infty n!\langle f_n,g_n\rangle_{L^2(\TT^d,\RR^d)^{\otimes n}},\quad f,g\in\FK\,.
    \end{align}

\subsection{The truncated equation and its generator}
In this section, we recall the basic properties of the solution $u^N$ to \eqref{eq:scaled-NS} from \cite{JP24} and derive its generator on the Fock space $\FK$. 

First we write \eqref{eq:scaled-NS} in the weak fomulation. 
For $\varphi \in \fS(\TT^d,\RR^d)\cap\HH$ and $t\geq0,$
\begin{equation}\label{3.1}
     u^{N}_t(\varphi) = u_0^{N}(\varphi) + \int_0^t u^{N}_s(\Delta\varphi)ds - \lambda_N\int_0^tB^{N}(u^{N}_s)(\varphi)ds + M_t^{\varphi},
\end{equation}
where $M_t^\varphi = \sqrt{2}\xi(\1_{[0,t]}\otimes (-\Delta)^{\frac12}\varphi)$ is a continuous martingale with quadratic variation 
$\langle M^\varphi\rangle_t = 2t\|(-\Delta)^{\frac12}\Lp\varphi\|_{L^2(\TT^d,\RR^d)}^2 \,,$ and 
\begin{align*}
	B^N(u)(\varphi) =& \rho^N*\Lp \operatorname{div}((\rho^N*u)\otimes (\rho^N*u))(\varphi)\\
	=&\iota2\pi\sum_{k_{1:2}\in\ZZ_0^d}\mathcal{R}_{k_1,k_2}^N((k_1+k_2)\cdot\hat{u}(k_1))(\hat{\varphi}(-k_1-k_2)\cdot\hat{u}(k_2))\,,
\end{align*}
where $\mathcal{R}_{k_1,k_2}^N:=\hat{\rho}^N(k_1)\hat{\rho}^N(k_2)\hat{\rho}^N(k_1+k_2)$. 
For the sake of technical simplification (especially to prove Lemma \ref{lem5.9}), we keep the assumption from \cite{CGT24} that
for $d\geq3,$ $N\in \NN+\frac12,$ and the norm in $\fR^N$ is $|\cdot|_\infty$, while for $d=2$ is Euclidean norm $|\cdot|$. 
By \cite[Theorem 3.4]{JP24}, there exists a unique solution $u^N$ to \eqref{3.1} for any divergence-free initial condition, and it is a strong Markov process with the invariant measure $\mu$ defined in Subsection \ref{sec:2.2}.
{Denote the distribution of $u^N$ by $\bP$ (with corresponding expectation $\bE$)} and its genertor by $\fL^N= \fL_0 + \fA^N$.
In the following, we write the generator $\fL^N$ on Fock space $\FK$ explicitly. 
By the isometry,
we only need to determine $\fF(\fL^N\varphi)(l_{1:n},k_{1:n})$ for $\varphi\in\FK_n$. Note that it must be symmetric, i.e. for all permutation $\sigma$, $\fF(\fL^N\varphi)(l_{
\sigma(1:n)},k_{\sigma(1:n)})=\fF(\fL^N\varphi)(l_{1:n},k_{1:n})$.
\begin{theorem}\label{thm.Geneator_in_Fourier}
For $N\in\NN$, let $\fL^N = \fL_0+\fA^N$. Then $\fA^N = \fA^N_+ + \fA^N_-$ with $\fA^N_-=-(\fA^N_+)^{*}$\,. $\fL_0(\FK_0) = \fA_+^N(\FK_0) = \fA_-^N(\FK_1) = 0$.
For $n\ge1$, the action of the operators $\fL_0,\,\fA_+^N$ and $\fA_-^N$ on $\varphi_n\in \FK_n$ is given by
\begin{align}
	\fF(\fL_0\varphi_n)(l_{1:n},k_{1:n}) =& -(2\pi)^{2}\sum_{i=1}^n|k_i|^{2}\hat{\varphi}_n(l_{1:n},k_{1:n}),\label{3.9}\\
	\fF(\fA_+^{N}\varphi_n)(l_{1:n+1},k_{1:n+1}) =&\frac{\lambda_N 2\pi\iota}{n+1}\sum_{1\leq i,j\leq n+1,i\neq j}\fR_{k_{i},k_j}^N\left(\hat\Lp(k_i)(k_i+k_{j})\right)(l_i)\times\nonumber\\
	&\times\left(\hat\Lp(k_{j}) \hat{\varphi}_{n}((\cdot,k_i+k_{j}),(l_{1:n+1\backslash i,j},k_{1:n+1\backslash i,j}))\right)(l_j),\label{3.10}\\
	\fF(\fA_-^{N}\varphi_n)(l_{1:n-1},k_{1:n-1}) =&\lambda_N 2\pi\iota n\sum_{j=1}^{n-1}\sum_{p+q=k_j}\fR_{p,q}^N\sum_{i,t=1}^dk_j^i{\hat{\Lp}_{l_j,t}(k_j)}\hat{\varphi}_{n}((t,p),(i,q),(l_{1:n-1\backslash j},k_{1:n-1\backslash j})),\label{3.11}
\end{align}
where $k_j^i$ is the $i$-th component of $k_j$.
Moreover, for $i=1,\dots,d,$ the momentum operator $M_i$, defined for $f\in\FK_n$ as $\fF(M_if)(l_{1:n},k_{1:n})=(\sum_{j=1}^nk_j^i)\hat{f}(l_{1:n},k_{1:n}),$ commutes with $\fL_0,\fA^N_{+},\fA^N_{-}$, i.e.
\begin{align}\label{eq:commute}
\fL^NM_i-M_i\fL^N=0 \,.
\end{align}
\end{theorem}
    
  \begin{remark}
     {Theorem \ref{thm.Geneator_in_Fourier} differs from (3.10) and (3.11) in \cite{JP24}. Since the functions in the Fock space are divergence-free, $\rho^N_{i,y}$ in \cite[(3.13)]{JP24} should be $\Lp\rho^N_{i,y}$. We give a new proof of Theorem \ref{thm.Geneator_in_Fourier} in Appendix \ref{appenA}. For the operator $\fA^N_+$, the Leray projection matrix associated with the single momentum $k_j$ acts on the first component of $\hat\varphi_n$, which is different from the case of stochastic Burgers equation in \cite{CGT24}.
    	This modification requires more technical calculation in the proof of the Replacement Lemma for $d=2$ and proof of Proposition \ref{prop:limite} for $d\geq3.$  }
    \end{remark}
{In this case, we can also prove the so-called graded sector condition for the operator $\fA^N$ by a similar argument as \cite{CGT24}. For the completeness of the paper, we put its proof in Appendix \ref{appenA}.} Before stating it, we introduce the so-called number operator.
\begin{definition}\label{def:num}
	The number operator $\fN:\FK\to\FK$ is defined by $\fN\psi=n\psi$ for $\psi\in\FK_n$.
\end{definition}

\begin{lem}[Graded Sector Condition]\label{lem:sector}
        For $d\geq2$ and $\varphi\in\FK$, one has
        \begin{equation}\label{eq:sector}
            \|(- \fL_0)^{-\frac{1}{2}}\fA_{\sigma}^{N}\varphi\|^2
            \lesssim_d\lambda^2 \|\sqrt{\fN}(-\fL_0)^{\frac{1}{2}}\varphi\|^2,
        \end{equation}
for $\sigma\in\{+,-\}.$ Furthermore, for any smooth function $\varphi$ in $\FK_2$, it holds that
\begin{align}\label{eq:A-}
	\lim_{N\to\infty}\|(- \fL_0)^{-\frac{1}{2}}\fA_{-}^{N}\varphi\|^2=0\,.
\end{align}
    \end{lem} 

\section{{The Fluctuation-Dissipation Theorem}}\label{sec3}
In this section, we revisit the Fluctuation-Dissipation Theorem from \cite{CGT24}, which is crucial for proving Theorem \ref{thm:main}.
We begin by establishing the tightness of the sequence $\{u^N\}_{N\geq1}$. 
{A key tool for this purpose is the so-called \emph{It\^o trick}.
\begin{lem}[It\^o trick]\label{lem3.1}
	Let $d\geq2$, and let $u^N$ be the solution to \eqref{eq:scaled-NS} with $\lambda_N$ given as in \eqref{lambda}. For any $p\geq2, T>0, n\geq1$ and $F\in\oplus_{k=1}^n\fH_k$, it holds that
	\begin{align*}
		\left[\bE\sup_{t\in[0,T]}\left|\int_0^TF(u^N_s)\ud s\right|^p\right]^{\frac1p}\lesssim_{n,p}T^{\frac12}\|(-\fL_0)^{-\frac12}F\|\,.
	\end{align*}
\end{lem}
We refer the readers to \cite[Lemma 2]{GJ13} for a proof. Using It\^o trick and a similar argument as \cite[Proposition 3.2]{CGT24}, we can deduce the following.
\begin{theorem}\label{thm:tight}
	The sequence $\{u^N\}_{N\geq1}$ is tight in $C([0,T],\fS'(\TT^d,\RR^d)).$
\end{theorem}

{Having established tightness, we aim to show that the law of each limit point of $u^N$ coincides with that of the solution to \eqref{eq:limeq}, which is characterized uniquely by the following martingle problem.
\begin{definition}\label{def:mart}
	Let $d\geq2,\, T>0,\, \Omega = C([0,T],\fS'(\TT^d,\RR^d))$, and $\fB$ the canonical Borel $\sigma$-algebra on $\Omega$. Let $\mu$ be a mean-zero and divergence-free space white noise on $\TT^d$ and $D$ be the constant in \eqref{eq:limeq}.
	We say that a probability measure $\bP$ on $(\Omega,\fB)$ solves the martingale problem for $(1+D)\fL_0$
	with initial distribution $\mu$, if for all
	$\varphi\in\fS(\TT^d,\RR^d)$, the canonical process $u$ under $\bP$ is such that
	\begin{align*}
		F_j(u_t)-F_j(\mu)-\int_0^t(1+D)\fL_0F_j(u_s)\ud s,\quad j=1,2
	\end{align*}
is a local martingale, where $F_1(u):=u(\varphi)$ and $F_2(u) := u(\varphi)^2-\|\varphi\|_{L^2(\TT^d,\RR^d)}^2$.
\end{definition}
Following the proof of \cite[Theorem D.1]{MW17}, we can show that the martingale problem for $(1+D)\fL_0$ in Definition \ref{def:mart} has a unique solution and uniquely characterises the law of the solution to \eqref{eq:limeq}. Therefore, we only need to porve each limit point of $\{u^N\}_{N\geq1}$ solves this martingale problem.  
The key ingredient for this is the following Fluctuation–Dissipation Theorem.}

\begin{theorem}[The Fluctuation-Dissipation Theorem]\label{thm:F-D}
	Let the divergence-free functions $\varphi,\,\psi\in\fS(\TT^d,\RR^d)$ be fixed, and define 
	\begin{align}\label{eq:fi}
		f_1:=\varphi,\quad f_2:=[\psi\otimes\varphi]_{sym}=\frac12(\varphi\otimes\psi+\psi\otimes\varphi)\,.
	\end{align}
	Then for $i=2,3$, there exists a family $\mathcal V^i=\{v^{N,n}\in\oplus_{j=i}^n\FK_j, n\in\NN\}$ and a positive constant $D$ (which is independent of $i$) such that
	\begin{align}
		&\lim_{n\to\infty}\limsup_{N\to\infty}\|v^{N,n}\|=0,\label{5.1}\\
		&\lim_{n\to\infty}\limsup_{N\to\infty}\|(-\fL_0)^{-\frac12}(-\fL^Nv^{N,n}-\fA^N_+f_{i-1}+\fA_-^Nv_i^{N,n})\|=0,\label{5.2}\\
		&\lim_{n\to\infty}\limsup_{N\to\infty}\|(-\fL_0)^{-\frac12}(\fA_-^Nv_i^{N,n}-D\fL_0 f_{i-1})\|=0\,.\label{5.3}
	\end{align}
\end{theorem}
{In the following sections, the main task is to establish the existence of functions $\{v^{N,n}\}_{N\geq1,n\in\NN}$ satisfying conditions \eqref{5.1}-\eqref{5.3} in Theorem \ref{thm:F-D}. 
Once this theorem is proved, Theorem~\ref{thm:main} follows by the same arguments as in Subsection~3.2 of \cite{CGT24}, combined with~\eqref{eq:A-}.
}

\section{Proof of Theorem \ref{thm:main} when $d\geq3$}\label{sec4}

We follow the route in \cite{CGT24} to directly verify that the solution to 
\begin{align}\label{eq:geneq}
	-\fL^N_{i,n}u^{N,n}=\fA^N_+f_{i-1}
\end{align}
satisfies conditions \eqref{5.1}–\eqref{5.3}, where $\fL^N_{i,n}=P_{i,n}\fL^NP_{i,n}$ with $P_{i,n}$ the orthogonal projection onto $\oplus_{j=i}^n\FK_j.$
{Compared with \cite{CGT24}, the main additional difficulty comes from the Leray projection and inner-product in the action of $\fA^N_{\pm}$ in \eqref{3.10} and \eqref{3.11}, which makes the vector components coupled and leads to several distinct cases for the divergence-free basis in $\RR^d$ during iteration in the proof of Lemma \ref{lem5.8} below.
In subsection \ref{sec4.2}, we discuss all these cases carefully and use the anti-symmetry property of $\fA^N$ to get the vector components in the limit are decoupled.
First, we recall some notations from \cite{CGT24}. 
}
\subsection{Notations and main proposition}\label{sec4.1}

Let $2\leq n\in\NN$ be fixed throughout the section.
For $N>0$, $i=2,3$ and $\sigma\in\{+,-\}$, define
\begin{align}
	T^{N,\sigma}&:= (-\fL_0)^{-1/2}(\fA^N_{\sigma})(-\fL_0)^{-1/2}\,,\nonumber\\
	T^N_{i,n}&:= (-\fL_0)^{-1/2}(\fA^N_{i,n})(-\fL_0)^{-1/2}\,,\nonumber\\
	T^{N,\sigma}_{i,n}&:= (-\fL_0)^{-1/2}(\fA^{N,\sigma}_{i,n})(-\fL_0)^{-1/2},\label{4.2}
\end{align}
where $\fA^N_{i,n}=P_{i,n}\fA^NP_{i,n},$ and $\fA^{N,\sigma}_{i,n}=P_{i,n}\fA_\sigma^NP_{i,n}.$
Since $n$ is fixed, by Lemma \ref{lem:sector},
$T^{N,\sigma}_{i,n}$ is bounded in $\FK$ uniformly in $N$,
and is skew-Hermitian. 
For $m\ge1$ and $a\ge 0$, define $\Pi^{(n)}_{a,m}$ be the set
of simple random walk paths $p=(p_0,\dots,p_a)$ such that $p_0=1$ and $p_a=m$, which do not reach height $n+1$ and can only reach $1$ at the endpoints.  
Given $p\in \Pi^{(n)}_{a,m}$, let $|p|:=a$ and
$\mathcal T^N_p$ be defined according to 
\begin{align}\label{4.19}
	\fT^N_p:= T^{N,\sigma_a}T^{N,\sigma_{a-1}}\cdots T^{N,\sigma_1},
\end{align}
where $\sigma_j=+$ (resp.$-$) if $p_{j}-p_{j-1}=1$ (resp.$-1$). If $|p|=0$, $\mathcal T^N_p:= 1$.
For $n\geq1$, recall the definition of $L^2_n$ in \eqref{L^2_n}. 
Similarly as \cite{CGT24}, we rewrite the action of the operators $\fA^N_+,\,\fA^N_-$ as follows: for $f\in L^2_n$,
\begin{align}
	&\fA_+^{N}f=\frac{1}{n+1}\sum_{1\leq i\neq i'\leq n+1}\fA_+^{N}[(i,i')]f,\quad
	{\fA_-^Nf}=n\sum_{q=1}^{n-1}\fA_-^{N}[q]f \,.\label{A_-}
\end{align} 
Here for $1\leq i\neq i'\leq n+1$, $1\leq q\leq n-1,$
\begin{align}
		\fF(\fA_+^{N}[(i,i')]f)(l_{1:n+1},k_{1:n+1}) :=&{\lambda_N 2\pi\iota}\fR_{k_{i},k_{i'}}^N(\hat\Lp(k_i)(k_i+k_{i'}))(l_i)\times\nonumber\\
		&\times\left(\hat\Lp(k_{i'})
		\hat f((\cdot,k_i+k_{i'}),(l_{1:n+1\backslash i,i'},k_{1:n+1\backslash i,i'}))\right)(l_{i'})\,,\nonumber\\
		\fF(\fA_-^{N}[q]f)(l_{1:n-1},k_{1:n-1}) :=&\lambda_N2\pi\iota \sum_{l+m=k_q}\fR_{l,m}^N
		\sum_{i,t=1}^dk_q^i{\hat{\Lp}_{l_q,t}(k_q)}\hat f((t,l),(i,m),(l_{1:n-1\backslash q},k_{1:n-1\backslash q}))\,.\label{Aq}
\end{align}
Let $T^{N,+}[(i,i')]$ and $T^{N,-}[q]$ be defined as in \eqref{4.2} with $\fA^N_+, \fA^N_-$ replaced by $\fA^N_+[(i,i')]$ and $\fA^N_-[q]$ respectively.
Arguing as the proof of \cite[Lemma 2.14]{CGT24}, we can derive that  
\begin{align}\label{e:TboundNew}
	\|T^{N,+}[(i,i')] f\|_{L^2_{n+1}}\vee \|T^{N,-}[q] f\|_{L^2_{n-1}}\lesssim\|f\|_{L^2_n}
\end{align}
for every $f\in L^2_n$.
Then $\mathcal{T}_p^N$ in \eqref{4.19} can be written in terms of $\fA_+^{N}[(i,i')],\;\fA_-^{N}[q].$ 

{For $f\in L^2_n$, let $\kappa(f):=n$. Recall that $\{\sigma_{k,1},\cdots,\sigma_{k,d-1}: k\in\ZZ_0^d\}$ given by  \eqref{vector-fields} is a CONS of $\HH$.
For $j_1,j_2\in\ZZ_0^d$ and $t_1,t_2\in\{1,\cdots,d-1\},$
let $\bj$ be either $j_1$ or $j_{1:2}$, and $\bt$ be either $t_1$ or $t_{1:2}$.
We denote $\sigma_{\bj,\bt}$ by $\sigma_{j_1,t_1}$ if $\kappa:=\kappa(\sigma_{\bj,\bt})=1$ or
$$\sigma_{j_{1:2},t_{1:2}}:=\frac{\sigma_{j_1,t_1}\otimes \sigma_{j_2,t_2}+\sigma_{j_2,t_2}\otimes \sigma_{j_1,t_1}}{2},$$
if $\kappa=2.$}
For $p\in\Pi_{a,m}^{(n)}$ with $a\geq0$ and $m\geq1$,
\begin{align}
	\mathcal{T}_p^N\sigma_{\bj,\bt}
	=\frac{1}{C_{m,\kappa}}\sum_{g\in\fG^\kappa[p]}T^{N,\sigma_a}[g_a]\cdots T^{N,\sigma_1}[g_1]\sigma_{\bj,\bt}
	=:\frac{1}{C_{m,\kappa}}\sum_{g\in\fG^\kappa[p]}\mathcal{T}_p^N[g]\sigma_{\bj,\bt}\,,\label{e:Tg}
\end{align}
where $C_{1,\kappa}=1$, $C_{m,\kappa}=\prod_{j=1}^{m-1}(\kappa+j)$ for $m>1$, and $\fG^\kappa[p]$ is a set whose elements $g=(g_s)_{s=1,\dots,a}$ are of the form
\begin{equ}
	g_s:=
	\begin{cases}
		(i_s,i'_s)\,,&\text{if $\sigma_s=+$,}\\
		q_s\,,&\text{if $\sigma_s=-$,}
	\end{cases}
\end{equ}
for $1\le i_s\neq i'_s,q_s\le p_{s}+\kappa-1$. 
{In \cite{CGT24}, the authors introduced a graphical representation for $g\in\fG^\kappa[p]$, see \cite[Fig.1]{CGT24} for an example.
The graph associated with $g = (g_s)_{s=1,\dots,a}$ has
$ a+1 $ columns, labeled from $0$ to $a$, where the $s$-th column contains $p_s + \kappa - 1$ vertices, numbered increasingly from top to bottom.
Edges are drawn only between consecutive columns, with each vertex connected to at most two vertices in the next column. 
If $g_s = (i_s,i'_s)$, the top vertex in column $s-1$, referred to as a branching point,
connects to the vertices with labels $i_s$ and $i'_s$ in column $s$; 
if $g_s = q_s$, the vertex with label $q_s$ in column $s$, referred to as a merging point,
connects to the vertices with labels $1$ and $2$ in column $s-1$. 
A path $\pi$ in the graph associated with $g$ is a sequence of connected vertices where $\pi(j)$ denotes the label of the vertex that $\pi$ encounters in column $j$.
We denote by $\widetilde \fG^2[p]$ the subset of $\fG^2[p]$ consisting of those $g$ whose associated graph has two connected components, one of which is a path with no branching and merging points.
By \cite[Lemma 2.18]{CGT24}, there exists a correspondence between $\fG^1[p]$ and $\tilde\fG^2[p]$. For $g\in\fG^1[p]$, we say that $\hat g$ and $\tilde g$ in $\tilde\fG^2[p]$ are the elements associated with $g$ if the graph associated with $g$ can be obtained from that of $\hat g$ and $\tilde g$ by removing their connected component corresponding
to the single path $\pi$.
}

Now we state the main proposition for the proof of Theorem \ref{thm:F-D}. To the end, we introduce the following notation.
Let $a\geq1,\, 1\leq m\leq a\wedge n$, $p\in\Pi_{a,m}^{(n)}$ and $g\in\fG^1[p]$ or in $\tilde\fG^2[p]$. 
Let $m_\kappa = m+\kappa-1$, where $\kappa=1$ if $g\in\fG^1[p]$ and 2 in the other case. Moreover, let\begin{align}
	{\tilde \sigma}_{\bj,\bt}:=
	\begin{cases}
		\sigma_{j_1,t_1}\,, & \text{ if } \bj=j_1,\bt=t_1, \\
		\sigma_{j_1,t_1}\otimes \sigma_{j_2,t_2}\,&  \text{ if } \bj=j_{1:2},\bt=t_{1:2}\,.
	\end{cases}
\end{align}
Then $\fT^N_p[g] {\tilde\sigma}_{\bj,\bt}\in L^2_{m_\kappa}$.
In the next subsection, we will prove the following result.
	\begin{prop}\label{prop:limite}
		Let $a\geq 1$ be even.  
		Then, there exists a constant $c(a,\lambda)\in\RR$ independent of $\bj,\,\bt$ such that for every $\bj',\,\bt'$,
		\begin{align}
			\lim_{N\to \infty}\sum_{p\in \Pi^{(n)}_{a,1}}\sum_{g\in\fG^1[p]}\langle{\tilde\sigma}_{-\bj',\bt'},\fT^N_p[g]{\tilde\sigma}_{\bj,\bt}\rangle_{L^2_{m_\kappa}}&=c(a,\lambda)\1_{\bj'=\bj}\1_{\bt'=\bt},\label{eq:liminner1}\\
			\lim_{N\to \infty}\sum_{p\in \Pi^{(n)}_{a,1}}\sum_{g\in\tilde\fG^2[p]}\langle{\tilde\sigma}_{-\bj',\bt'},\fT^N_p[g]{\tilde\sigma}_{\bj,\bt}\rangle_{L^2_{m_\kappa}}&=\frac{|j_1|^2}{|\bj|^2}c(a,\lambda)\left(\1_{\bj'=\bj}\1_{\bt'=\bt}+\1_{\bj'=\bj^T}\1_{\bt'=\bt^T}\right),\label{eq:liminner2}
		\end{align}
		where $\bj^T=(j_2,j_1)$, and $\bt^T=(t_2,t_1)$.
	\end{prop}

{By Proposition \ref{prop:limite} and similar arguments as \cite{CGT24}, we can prove Theorem \ref{thm:F-D}. We put details in Appendix \ref{appendixB}.}

\subsection{Proof of Proposition \ref{prop:limite}}\label{sec4.2}

{The main difficulty in proving Proposition \ref{prop:limite} comes from the Leray projection and vector-valued nature of the equation \eqref{eq:scaled-NS}. In the following, we decompose the action of Leray projection appearing in $\fA^N_+$ and $\fA_-^N$ in \eqref{3.10} and \eqref{3.11} on a vector into two parts: the vector itself and a new vector times its inner product with the original one.
This decomposition produces multiple cases for the basis element	$a_{j_1,t_1}$ and leads to mixed terms involving $a_{j_1,t_1}$ and $a_{j_1,t'_1}$ in $\lim_{N\to \infty}\langle{\tilde\sigma}_{-j'_1,t'_1},\fT^N_p[g]{\tilde\sigma}_{j_1,t_1}\rangle_{L^2_{m_\kappa}}$, whereas the desired limiting decoupled stochastic heat equation corresponds to $\1_{t_1=t'_1}$ multiplied by a constant.
To address these problems, 
we perform a case-by-case analysis in Lemma \ref{lem5.8} depending on how the Leray projection acts, in order to track the change of $a_{j_1,t_1}$.
In the proof of Proposition \ref{prop:limite}, we use the Hermitian property of $\sum_{p\in\Pi_{a,1}^{(n)}}\fT_p^N$, and perform a rotational change of coordinates in the integral to get the desired symmetry in $t_1$ and $t'_1.$ 
}

{In \cite{CGT24}, the authors introduced $\fA^{N,\delta}_{\pm}$ for a uniform integrability type condition. In this case, due to 
Leray projection, we have to replace $\fA^{N,\delta}_{\pm}$ by the following operators on $L^2(\TT^d,\RR^d)^{\otimes n}.$
For fixed $j_1\in\ZZ_0^d$, $\delta\geq1,$ $f\in L^2(\TT^d,\RR^d)^{\otimes n},$ $1\leq i\neq i'\leq n+1$, and $1\leq q\leq n-1,$ set
\begin{align*}
	\fF(\fA_+^{N,\delta}[(i,i')]f)(l_{1:n+1},k_{1:n+1}) :=&\lambda N^{-\frac d2} | N2\pi\iota|^{\delta}\fR_{k_{i},k_{i'}}^N|k_i+k_{i'}|^{\delta}\frac{\sum_{u=1}^{d-1}a_{j_1,u}(l_i)}{\sqrt{d-1}}\frac{\sum_{u=1}^{d-1}a_{j_1,u}(l_{i'})}{\sqrt{d-1}}\times\\
	&|\hat f((\cdot,k_i+k_{i'}),(l_{1:n+1\backslash i,i'},k_{1:n+1\backslash i,i'}))|,\\
	\fF(\fA_-^{N,\delta}[q]f)(l_{1:n-1},k_{1:n-1}) :=&\lambda N^{-\frac d2} |N2\pi\iota|^{\delta}
	\sum_{l+m=k_q}\fR_{l,m}^N|k_q|^{\delta}\left|\hat f((l_q,l),(\cdot,m),(l_{1:n-1\backslash q},k_{1:n-1\backslash q}))\right|\,,
\end{align*}
and \begin{align*}
	{T^{N,\delta,+}}[(i,i')]:=(-\fL_0)^{-\frac{\delta}{2}}\fA_+^{N,\delta}[(i,i')](-\fL_0)^{-\frac{\delta}{2}},\quad
	{T^{N,\delta,-}}[q]:=(-\fL_0)^{-\frac{\delta}{2}}\fA_-^{N,\delta}[q](-\fL_0)^{-\frac{\delta}{2}}.
\end{align*}
We also set $\fT_p^{N,\delta}[g]:=T^{N,\delta,\sigma_a}[g_a]\cdots T^{N,\delta,\sigma_1}[g_1]$. 
Arguing as the proof of \cite[Lemma 2.14]{CGT24}, we can derive that  
\begin{align}\label{e:delta}
	\|T^{N,\delta,+}[(i,i')] f\|_{L^2_{n+1}}\vee \|T^{N,\delta,-}[q] f\|_{L^2_{n-1}}\lesssim_\delta\|f\|_{L^2_n}
\end{align}
for $f\in (L^2(\TT^d,\RR^d))^{\otimes n},$ and $1\leq\delta<d/2$.
The proof for the uniform integrability of the integrand in $\langle{\tilde\sigma}_{-\bj',\bt'},\fT^N_p[g]{\tilde\sigma}_{\bj,\bt}\rangle_{L^2_{m_\kappa}}$ proceeds in two steps.
First, we will show that the integrand is bounded by that appearing in $\langle{\tilde\sigma}_{-\bj',\bt'},\fT^{N,\delta}_p[g]{\tilde\sigma}_{\bj,\bt}\rangle_{L^2_{m_\kappa}}$ for $\delta=1$. This requires a comparison between the Fourier transforms of the kernels of $\fT^N_p[g]{\tilde\sigma}_{\bj,\bt}$ and $\fT^{N,1}_p[g]{\tilde\sigma}_{\bj,\bt}$ (see Lemma \ref{lem5.8} below).
The second step is to prove the uniform-integrability of this latter integrand by using operators $\fT_p^{N,\delta}[g]$ for $\delta>1$.}

In the following, we first concentrate on the Fourier transform of the kernel of $\fT^N_p[g]{\tilde\sigma}_{\bj,\bt}$, which we denote as 
$f_g^N(l_{1:m_\kappa},k_{1:m_\kappa};\bj,\bt).$ It has the form
\begin{align}\label{eq:kernelf}
	f_g^N(l_{1:m_\kappa},k_{1:m_\kappa};\bj,\bt)= N^{\frac d2(1-m)}F^N_g(l_{1:m_\kappa},N^{-1} k_{1:m_\kappa};N^{-1} \bj,\bt)\,.
\end{align} 
Also, for $\delta\geq1$, let $f^{N,\delta}_g$ be the Fourier transform of $\fT_p^{N,\delta}[g]\tilde\sigma_{\bj,\bt}$, and $F_g^{N,\delta}(l_{1:m_\kappa},N^{-1}k_{1:m_\kappa};N^{-1}\bj,\bt)$ be defined similarly through \eqref{eq:kernelf}.
In the following lemma, we derive expressions for $F^N_g$ and $F_g^{N,\delta}$ in terms of ratios of polynomials, and compare the terms in $F^N_g$ and $F_g^{N,1}$.
{Since the solution to \eqref{eq:scaled-NS} is vector-valued and divergence-free, the operator $\fA^N_+$ in \eqref{3.10} 
involves the Leray projection acting on vectors, which distinguishes our proof from the stochastic Burgers equation case in \cite[Lemma 2.19]{CGT24}. As mentioned before, the Leray projection in \eqref{3.10} has two effects: it either preserves the vector itself or generates a new vector times its inner product with the original one.
Thus, in order to track the evolution of the basis element $a_{j_1,t_1}$, we need to discuss different cases during the iteration.
Also, the operator $\fA^N_-$ in \eqref{3.11} has the Leray projection on the first vector component, and an inner-product on the second component, which means this action could
preserve a vector, produce a new one via inner product with the original vector, or yield a scalar by taking the inner product between two vectors. Hence, it also requires a case-by-case analysis.
To this end, compared with \cite{CGT24}, we introduced extra $G_g$ term to track the change of $a_{j_1,t_1}$ during the iteration.
}

\begin{lem}\label{lem5.8}
	Let $a\geq1,\, 1\leq m\leq a\wedge n$, $p\in\Pi_{a,m}^{(n)}$, $\delta\geq1$, and
	$M=(1+a-m)/2$ be the number of $T^{N,-}$ $(T^{N,\delta,-})$ in the product $\fT^N_p$ $(\fT^{N,\delta}_p)$. 
	For $g\in\tilde\fG^2[p]$, the kernel $F_g^N$, defined in \eqref{eq:kernelf}, can be expressed as 
	\begin{align}
		&F_g^N(l_{1:m+1},x_{1:m+1};\bj,\bt)\nonumber\\
		=&\1_{\sum_{j\neq \pi(a)} x_j= j_1}\1_{x_{\pi(a)}=j_2} 
		\frac{|j_1|}{|\bj|}\frac{(\lambda\iota)^a }{|x_{1:m+1}|}a_{j_2,t_2}^{l_{\pi(a)}}\sum_{\substack{y_{1:M}\in (\frac1N \ZZ^d_0)^M,\\|y_i|_\infty\le 1, 1\leq i\le M}} \frac{N^{-dM}I_g(y_{1:M};x_{1:m+1\setminus\{\pi(a)\}})}{Q_g(y_{1:M};x_{1:m+1})}
		\times\nonumber\\
		&P_g(l_{1:m+1\backslash\{c_1b_1,\pi(a)\}},y_{1:M};x_{1:m+1\setminus\{\pi(a)\}},\frac{j_1}{|j_1|})G_g(y_{1:M};x_{1:m+1\setminus\{\pi(a)\}},a_{j_1,t_1})\,
		\label{eq:formkernelG},
	\end{align}
	where $|x_{1:m+1}| = \sqrt{\sum_{i=1}^{m+1}|x_i|^2}$, and
	\begin{align}\label{*}
		G_g=&c_1(a_{j_1,t_1}+R^1_g)(l_{b_1})+(1-c_1)R^2_g \,.
	\end{align}
	Here $b_1$ belongs to $1:m+1\backslash \{\pi(a)\}$, $c_1=0$ or 1, and
	\begin{enumerate}[noitemsep]
		\item  [(i)] $I_g$ is a product of  indicator functions, each imposing that certain linear combinations of its arguments have $|\cdot|_\infty$ in $(0,1]$,
		
		{\item [(ii)] $Q_g$ is a homogenous real-valued polynomial of degree $2(a-1)$ 
		in $y_{1:M};x_{1:m+1}$, and satisfies $$Q_g(y_{1:M};x_{1:m+1})=Q_g(Uy_1,\cdots,Uy_M;Ux_1,\cdots,Ux_{m+1})$$ for any unitary matrix $U$. Moreover, $Q_g\ge0$ and it does not vanish on the support of $I_g$,}
		
		{\item [(iii)] $P_g(l_{1:m+1\backslash\{c_1b_1,\pi(a)\}})=\prod_{k\in \{1:m+1\backslash\{c_1b_1,\pi(a)\}\}} P_g^{(k)}(l_k)$ is a real-valued function
		in $y_{1:M},$ $x_{1:m+1\backslash \{\pi(a)\}},$ $\frac{j_1}{|j_1|}$\,,
}		
		\item [(iv)] $R_g^1$ and $R_g^2$ are vector-valued and real-valued functions
		in $y_{1:M},x_{1:m+1\backslash\{\pi(a)\}},a_{j_1,t_1}$ respectively\,,
   
        \item [(v)] the terms in $R_g^1$, $R_g^2$ and $P_g^{(k)}$ for $k\in \{1:m+1\backslash\{c_1b_1,\pi(a)\}\}$ are in the form of the products of vectors in $V$ and $V'$ and inner-products with vectors in these sets, 
   	where $V$ is the set of vectors which are linear combinations of $y_{1:M},x_{1:m+1\backslash\{\pi(a)\}},\frac{j_1}{|j_1|}$, and $V'$ is the set of vectors in $V$ divided by its Euclidean norm\,,
 
		\item [(vi)] if $m=1$, then for almost every $y_{1:M}$, $Q_g(y_{1:M};0)>0$, and if $m\geq2$, the terms in $P_g(l_{1:m+1\backslash\{c_1b_1,\pi(a)\}},y_{1:M};x_{1:m+1\backslash\{\pi(a)\}},\frac{j_1}{|j_1|})$ and $G_g(y_{1:M};x_{1:m+1\backslash\{\pi(a)\}},a_{j_1,t_1})$ do not involve the vector $\sum_{i\neq\pi(a)}x_i$.
	\end{enumerate}
Similarly, the kernel $F_g^{N,\delta}$ can be written as  
	\begin{align}
		&F_g^{N,\delta}(l_{1:m+1},x_{1:m+1};\bj,\bt)\nonumber\\
		=&\1_{\sum_{j\neq \pi(a)} x_j= j_1}\1_{x_{\pi(a)}=j_2} 
		\frac{|j_1|^\delta}{|\bj|^\delta}\frac{\lambda^a }{|x_{1:m+1}|^\delta}a_{j_2,t_2}^{l_{\pi(a)}}\sum_{\substack{y_{1:M}\in (\frac1N \ZZ^d_0)^M:\\|y_i|_\infty\le 1, i\le M}} \frac{N^{-dM}((P_g^1)^\delta I_g)(y_{1:M};x_{1:m+1\setminus\{\pi(a)\}})}{Q_g^\delta(y_{1:M};x_{1:m+1})}
		\times\nonumber\\
		&\prod_{k\in\{1:m+1\backslash\{\pi(a)\}\}}\frac{1}{\sqrt{d-1}}\left(a_{j_1,1}+\cdots+a_{j_1,d-1}\right)(l_{k})\,.\label{eq:formkernelG1}
	\end{align}
	Here $I_g$ and $Q_g$ are the same as the ones in \eqref{eq:formkernelG}, and $P_{g}^1$ is the product of scalars depending on $y_{1:M},x_{1:m+1\setminus\{\pi(a)\}}$, satisfying that 
	\begin{align}\label{eq:bound}
		|P_g||G_g|\leq P^1_g.
	\end{align}
	For $c_1=1,$ $|G_g|$ denotes the norm of $a_{j_1,t_1}+R_g^1$. 
	
	Finally, if $j_1\in\ZZ_0^d,\, t_1\in\{1,\cdots,d-1\}$ and $g\in\fG^1[p]$, then the analog of \eqref{eq:formkernelG} $($resp.\eqref{eq:formkernelG1}$)$ holds upon setting $\bj=j_1,\, \bt=t_1$, removing the dependence of $j_2,\, t_2$ and $\pi(a)$, replacing both $x_{1:m+1}$ and $x_{1:m+1\backslash\{\pi(a)\}}$ with $x_{1:m}$ in \eqref{eq:formkernelG} $($resp.\eqref{eq:formkernelG1}$)$.
\end{lem}
\begin{proof}
	We only prove the case $\bj=j_{1:2}$ via induction on $|p|=a\ge1$. The argument for $\bj=j_1$ is similar. 
	For $a=1$ and $p\in \Pi^{(n)}_{1,m}$, 
	it follows that $m=2,\, M=0$, $\sigma_1=+$ and $g_1=(i,i')$ for some $1\le i\neq i'\le 3$. 
	By the definition of $T^{N,+}[(i,i')]$ and $a_{N^{-1}j,t}=a_{j,t}$ for every $j,t$, 
	we obtain that
	\begin{align}
		F^N_{g_1}(l_{1:3},k_{1:3};\bj,\bt)
		=&\frac{\lambda \iota}{2\pi|k_{1:3}|}\frac{|j_1|}{|j_{1:2}|}\fR^1_{k_i,k_{i'}}\1_{\sum\limits_{j\neq \pi(a)} k_j= j_1}\1_{k_{\pi(a)}=j_2}a_{j_2,t_2}^{l_{\pi(a)}}\left(\hat\Lp(k_i)\frac{j_1}{|j_1|}\right)(l_i)\left(\hat\Lp(k_{i'})a_{j_1,t_1}\right)(l_{i'})\,.\label{induc}
	\end{align}
	This expression is of the form \eqref{eq:formkernelG}, 
	with $Q_{g_1}\equiv 1$, $I_{g_1}=\fR^1_{x_{i},x_{i'}}$, 
	$c_1=1,$ $b_1=i',$ $\{1:m+1\backslash\{c_1b_1,\pi(a)\}\}=\{i\},$ 
	$P_{g_1}=\frac{1}{2\pi}\hat\Lp(x_i)\frac{j_1}{|j_1|}$,
	and $G_{g_1} = (a_{j_1,t_1}+R_{g_1}^1)(l_{i'})$ with
	$R_{g_1}^{1}=-\frac{x_{i'}\cdot a_{j_1,t_1}}{|x_{i'}|^2}x_{i'},$ 
	and $(i)-(v)$ hold. 
	From the definition of $T^{N,\delta,+}[(i,i')]$, 
	we also obtain that 
	\begin{align}
		F^{N,\delta}_{g_1}(l_{1:3},k_{1:3};\bj,\bt)
		=&\frac{\lambda }{(2\pi)^\delta|k_{1:3}|^\delta}\frac{|j_1|^\delta}{|j_{1:2}|^\delta}\fR^1_{k_i,k_{i'}}\1_{\sum\limits_{j\neq \pi(a)} k_j= j_1}\1_{k_{\pi(a)}=j_2}a_{j_2,t_2}^{l_{\pi(a)}}
		\frac{\sum_{u=1}^{d-1}a_{j_1,u}(l_i)}{\sqrt{d-1}}
		\frac{\sum_{u=1}^{d-1}a_{j_1,u}(l_{i'})}{\sqrt{d-1}}.\label{induc'}
	\end{align}
 Since $P_{g_1}^1=\frac{1}{2\pi}$, and $G_{g_1}=(\Lp(x_{i'})a_{j_1,t_1})(l_{i'})$, we have $|P_{g_1}||G_{g_1}|\leq P^1_{g_1}.$
	
	As for the induction step, we assume that the statement is true for a given $a\ge1$. Now we 
	take $p\in\Pi^{(n)}_{a+1,m},\, g\in\widetilde\fG^2[p]$ and then $\fT^N_p[g]=T^{N,\sigma}[g_{a+1}]\mathcal T_{p'}^N[g']$ for some 
	$p'\in\Pi^{(n)}_{a,m-\sigma_{a+1}}$ and $g'\in \widetilde\fG^2[p']$. 
	{\bfseries\noindent\textit{1.The case $\sigma_{a+1}=+$.}}
	Let $i,\, i'$ be such that 
	$g_{a+1}=(i, i')$. By the definition of $T^{N,+}[(i,i')]$, we have 
	\begin{align}
		&\fF(T^{N,+}[(i,i')]\mathcal T_{p'}^N[g']{\tilde\sigma_{\bj,\bt}})(l_{1:m+1},k_{1:m+1};\bj,\bt)\nonumber\\
		=&\frac{\lambda N^{1-\frac{d}{2}+\frac d2(2-m)}\iota}{(2\pi)|k_{1:m+1}|}\frac{\fR^N_{k_i,k_{i'}}}{\sqrt{|k_i+k_{i'}|^2+|k_{1:m+1\backslash i,i'}|^2}}\times\nonumber\\
		&\left(\hat\Lp(k_i)k_{i'}\right)(l_i)\left(\hat\Lp(k_{i'})F_{g'}^N((\cdot,N^{-1}(k_i+k_{i'})),(l_{1:m+1\backslash i,i'},N^{-1}k_{1:m+1\backslash i,i'});N^{-1}\bj,\bt)\right)(l_{i'}).\label{eq:induc_+}
	\end{align}
	Similarly, by the definition of $T^{N,\delta,+}[(i,i')]$, we have 
	\begin{align}
		&\fF(T^{N,\delta,+}[(i,i')]\mathcal T_{p'}^N[g']{\tilde\sigma}_{\bj,\bt})(l_{1:m+1},k_{1:m+1};\bj,\bt)\nonumber\\
		=&\frac{\lambda N^{\delta-\frac{d}{2}+\frac d2(2-m)}}{(2\pi)^\delta|k_{1:m+1}|^\delta}\frac{\fR^N_{k_i,k_{i'}}}{(|k_i+k_{i'}|^2+|k_{1:m+1\backslash i,i'}|^2)^{\frac\delta2}}|k_i+k_{i'}|^\delta\times\nonumber\\
		&\frac{\sum_{u=1}^{d-1}a_{j_1,u}(l_i)}{\sqrt{d-1}}\frac{\sum_{u=1}^{d-1}a_{j_1,u}(l_{i'})}{\sqrt{d-1}}|F_{g'}^N((\cdot,N^{-1}(k_i+k_{i'})),(l_{1:m+1\backslash i,i'},N^{-1}k_{1:m+1\backslash i,i'});N^{-1}\bj,\bt)|\,.
	\end{align}
	Recall that $\pi(a)$ is not a branching point. 
	The functions $I_g$ and $Q_g$ in $F_g^N$ and $F_g^{N,\delta}$ are the same:
	\begin{align*}
		I_g(y_{1:M};x_{1:m+1\setminus\{\pi(a+1)\}})=&\fR_{x_i,x_{i'}}^1I_{g'}(y_{1:M};x_i+x_{i'},x_{1:m+1\setminus\{i,i',\pi(a+1)\}})\,,\nonumber\\
		{Q_g(y_{1:M};x_{1:m+1})}=&(|x_i+x_{i'}|^2+|x_{1:m+1\backslash i,i'}|^2)Q_{g'}(y_{1:M};x_i+x_{i'},x_{1:m+1\setminus\{i,i'\}})\,.
	\end{align*}
	For the kernel $F_g^{N,\delta}$, we have
	\begin{align}
		P^1_g(y_{1:M};x_{1:m+1\backslash\{\pi(a+1)\}})=\frac{1}{2\pi}|x_i+x_{i'}|P^1_{g'}(y_{1:M};x_i+x_{i'},x_{1:m+1\setminus\{i,i',\pi(a+1)\}})\,.\label{induc_+1'}
	\end{align}
It remains to determine $P_g$ and $G_g$ for $F_g^N.$
 By \eqref{eq:induc_+}, the Leray matrix $\hat\Lp(k_{i'})$ acts on the first component of $F_{g'}^N$, which could be $P_{g'}^{(1)}$ or $G_{g'}$, corresponding to the following case 1.1 and case 1.2 respectively.
	
	{\bfseries\noindent\textit{1.1.The case $c'_1b'_1>1$ or $c'_1=0$.}}
	In this case, the first component of $F_{g'}^N$ is $P_{g'}^{(1)}$. Thus by \eqref{eq:induc_+}, we have
	\begin{align}
		P_g^{(i)}(y_{1:M}; x_{1:m-1\backslash\{\pi(a+1)\}},\frac{j_1}{|j_1|}) =& \frac{1}{2\pi}\hat\Lp(x_i)x_{i'},\nonumber\\
		P_g^{(i')}(y_{1:M}; x_{1:m-1\backslash\{\pi(a+1)\}},\frac{j_1}{|j_1|}) =& \,\hat\Lp(x_{i'})P_{g'}^{(1)},\nonumber\\
		\bigotimes_{k\in\{1:m+1\backslash\{i,i',c_1b_1,\pi(a+1)\}\}}P_{g}^{(k)}(y_{1:M}; x_{1:m-1\backslash\{\pi(a+1)\}},\frac{j_1}{|j_1|}) =& \bigotimes_{k'\in\{2:m\backslash\{c'_1b'_1,\pi(a)\}\}}P_{g'}^{(k')},\label{induc_+1}
	\end{align}
and $G_g$ is given by \eqref{*} with $c_1=c'_1$, $c_1b_1$ belonging to $(\{1:m+1\}\backslash\{i,i',\pi(a+1)\})\cup\{0\},$ and
\begin{align}
		R^i_g(y_{1:M}; x_{1:m-1\backslash\{\pi(a+1)\}},a_{j_1,t_1}) =& R^i_{g'} (y_{1:M};x_i+x_{i'},x_{1:m+1\setminus\{i,i',\pi(a+1)\}},a_{j_1,t_1}),\, i=1,2,\label{induc_+1.1}
\end{align}
	where $\bigotimes$ denotes the tensor product of vectors, and the omitted arguments of $P_{g'}^{(1)}$ and $P_{g'}^{(k')}$ are $(y_{1:M};x_i+x_{i'},x_{1:m+1\setminus\{i,i',\pi(a+1)\}},\frac{j_1}{|j_1|})$. 
	Since $\hat\Lp(x_i)x_{i'}=\hat\Lp(x_i)(x_i+x_{i'})$ and $|\hat\Lp(x)y|\leq|y|$ for all $x,\, y\in\ZZ_0^d$, by \eqref{induc_+1} it holds that
	\begin{align}
		|P_g||G_g|\leq\frac{1}{2\pi}|x_i+x_{i'}||P_{g'}||G_{g'}|\label{5.46}
	\end{align} 
	with the arguments omitted.
	By the induction hypothesis, the right hand side of the above equation is bounded by  $\frac{1}{2\pi}|x_i+x_{i'}|P^1_{g'}$, which implies \eqref{eq:bound} by \eqref{induc_+1'}.
	
	{\bfseries\noindent\textit{1.2.The case $c'_1b'_1=1$.}} 
	In this case, the first component of $F_{g'}^N$ is $G_{g'}$, which is given by
	$(a_{j_1,t_1}+R_{g'}^1)(l_1)=:\tilde G_{g'}(l_1).$
	Similarly to the previous case, we have
		\begin{align}
		P_g^{(i)}(y_{1:M}; x_{1:m-1\backslash\{\pi(a+1)\}},\frac{j_1}{|j_1|}) =& \frac{1}{2\pi}\hat\Lp(x_i)x_{i'},\nonumber\\
		\bigotimes_{k\in{1:m+1\backslash\{i,i',\pi(a+1)\}}}P_{g}^{(k)}(y_{1:M}; x_{1:m-1\backslash\{\pi(a+1)\}},\frac{j_1}{|j_1|})=&\bigotimes_{k'\in2:m\backslash\{\pi(a)\}}P_{g'}^{(k')},
	\end{align}
	and $G_g$ is given by \eqref{*} with $c_1=1$, $b_1=i'$ and
	\begin{align}
			R_g^1(y_{1:M}; x_{1:m-1\backslash\{\pi(a+1)\}},a_{j_1,t_1})=
		-\frac{x_{i'}\cdot a_{j_1,t_1}}{|x_{i'}|^2}{x_{i'}}+\hat\Lp(x_{i'})R_{g'}^1(y_{1:M};x_i+x_{i'},x_{1:m+1\setminus\{i,i',\pi(a+1)\}},a_{j_1,t_1}),\label{induc_+3.1}
	\end{align}
	where the omitted argument of $P_{g'}^{(k')}$ is $(l_{1:m+1\backslash\{i,i',c_1b_1,\pi(a+1)\}},y_{1:M};x_i+x_{i'},x_{1:m+1\setminus\{i,i',\pi(a+1)\}},\frac{j_1}{|j_1|})$.
	Similar as the previous case, since $|G_g|=|\Lp(x_{i'})\tilde G_{g'}|\leq|G_{g'}|$, 
	\eqref{eq:bound} holds.
	
	{\bfseries\noindent\textit{2.The case $\sigma_{a+1}=-$.}}
	In this case, there exists $q$ such that $g_{a+1}=q$. By the definition of $T^{N,-}[q]$, we have
 \begin{align}
		&F^N_{g}(l_{1:m-1},k_{1:m-1};\bj,\bt)\nonumber\\
		=&\frac{\lambda\iota \1_{0<|k_q|_\infty\leq1}}{2\pi|k_{1:m-1}|}
		\sum_{\substack{y\in N^{-1}\ZZ_0^d:\\|y|_\infty\leq1}}\1_{0<|k_q-y|_\infty\leq1}\frac{\sum_{i}k_q^i\left({\hat{\Lp}(k_q)}F^N_{g'}((\cdot,y),(i,k_q-y),(l_{1:m-1\backslash q},k_{1:m-1\backslash q}))\right)(l_q)}{\sqrt{|y|^2+|k_q-y|^2+|k_{1:m-1\backslash q}|^2}}.\label{eq:induc_-}
	\end{align}
	Similarly, by the definition of $T^{N,\delta,-}[q]$, we have
	\begin{align}
		&F^{N,\delta}_{g}(l_{1:m-1},k_{1:m-1};\bj,\bt)\nonumber\\
		=&\frac{\lambda}{(2\pi)^\delta|k_{1:m-1}|^\delta}\1_{0<|k_q|_\infty\leq1}|k_q|^\delta
		\sum_{\substack{y\in N^{-1}\ZZ_0^d:\\|y|_\infty\leq1}}\1_{0<|k_q-y|_\infty\leq1}\frac{\left|F^N_{g'}((l_q,y),(\cdot,k_q-y),(l_{1:m-1\backslash q},k_{1:m-1\backslash q}))\right|}{(|y|^2+|k_q-y|^2+|k_{1:m-1\backslash q}|^2)^{\frac\delta2}}.
	\end{align}
	Recall that $\pi(a)$ is not a merging point. 
	The functions $I_g$ and $Q_g$ are the same in $F_g^N$ and $F_g^{N,1}$:
	\begin{align*}
		I_g(y_{1:M};x_{1:m-1\setminus\{\pi(a+1)\}})=&\1_{0<|x_q|_\infty\leq1,0<|y_M-x_q|_\infty\leq1}I_{g'}(y_{1:M-1};y_M,x_q-y_M,x_{1:m-1\setminus\{q,\pi(a+1)\}})\,,\nonumber\\
		{Q_g(y_{1:M};x_{1:m-1})}=&(|y_M|^2+|x_q-y_M|^2+|x_{1:m-1\backslash q}|^2)Q_{g'}(y_{1:M-1};y_M,x_q-y_M,x_{1:m-1\setminus\{q\}})\,.
	\end{align*}
	For the kernel $F_g^{N,\delta}$, since $\frac{1}{\sqrt{d-1}}|a_{j_1,1}+\cdots+a_{j_1,d-1}|=1$, we have
	\begin{align}
		P^1_g(y_{1:M};x_{1:m-1\setminus\{\pi(a+1)\}})=&\frac{1}{2\pi}|x_q|P^1_{g'}(y_{1:M-1};y_M,x_q-y_M,x_{1:m-1\setminus\{q,\pi(a+1)\}}).\label{induc_-1'}
	\end{align}
	For $F^N_{g}$, there are several cases depending on the indices in $F_{g'}^N$. By \eqref{eq:induc_-}, $\hat\Lp(k_q)$ and $k_q$ act on the first and the second components of $F^N_{g'}$, which could be $P_{g'}^{(1)},\, P_{g'}^{(2)},$ or $G_{g'}$.
	Now we discuss the different cases.
	
	{\bfseries\noindent\textit{2.1.The case $c'_1b'_1>2$ or $c'_1=0$.}}
	In this case, the first and the second components of $F^N_{g'}$ are $P_{g'}^{(1)}$ and $P_{g'}^{(2)}$ respectively. Thus we have
	\begin{align}
	P_g^{(q)}(y_{1:M}; x_{1:m-1\backslash\{\pi(a+1)\}},\frac{j_1}{|j_1|}) =&\frac{1}{2\pi}\left(x_q\cdot P_{g'}^{(2)}\right)\hat\Lp(x_q)P_{g'}^{(1)} ,\nonumber\\
	\bigotimes_{k\in{1:m-1\backslash\{c_1b_1,q,\pi(a+1)\}}}P_{g}^{(k)}(y_{1:M}; x_{1:m-1\backslash\{\pi(a+1)\}},\frac{j_1}{|j_1|})=&\bigotimes_{k'\in3:m\backslash\{c'_1b'_1,\pi(a)\}}P_{g'}^{(k')},\label{induc_-1}
\end{align}
and $G_g$ is given by \eqref{*} with $c_1=c'_1$, $c_1b_1$ belonging to $(\{1:m-1\}\backslash\{q,\pi(a+1)\})\cup\{0\},$ and
\begin{align}
	R^i_g(y_{1:M-1};y_M,x_q-y_M,x_{1:m-1\setminus\{q,\pi(a+1)\}},a_{j_1,t_1}) =& R^i_{g'} (y_{1:M-1};x_{1:m\setminus\{\pi(a)\}},a_{j_1,t_1}),\, i=1,2,\label{induc_-1.1}
\end{align}
	where 
	the omitted arguments of $P_{g'}^{(k')}$ are $(y_{1:M-1};y_M,x_q-y_M,x_{1:m-1\setminus\{q,\pi(a+1)\}},\frac{j_1}{|j_1|})$. By the Cauchy-Schwarz inequality and $|\Lp(x)y|\leq|y|$ for all $x,\, y\in\ZZ_0^d$, \eqref{induc_-1} implies that
	\begin{align}
		|P_g||G_g|\leq\frac{1}{2\pi}|x_q||P_{g'}||G_{g'}|\label{5.55_1}
	\end{align} 
	with the argments omitted.
	By the induction hypothesis, the right hand side of the above is bounded by  $\frac{1}{2\pi}|x_q|P^1_{g'}$, which implies \eqref{eq:bound} by \eqref{induc_-1'}.
	
	{\bfseries\noindent\textit{2.2.The case $c'_1b'_1=1$.}} 
	In this case, the first component of $F_{g'}^N$ is $G_{g'}$, and the second one is $P_{g'}^{(2)}$. $G_{g'}$ is given by
	$(a_{j_1,t_1}+R_{g'}^1)(l_1)=:\tilde G_{g'}(l_1).$
	Thus we have
		\begin{align}
			\bigotimes_{k\in{1:m-1\backslash\{q,\pi(a+1)\}}}P_{g}^{(k)}(y_{1:M}; x_{1:m-1\backslash\{\pi(a+1)\}},\frac{j_1}{|j_1|})&=\frac{1}{2\pi}\left(x_q\cdot P_{g'}^{(2)}\right)\bigotimes_{k'\in3:m\backslash\{\pi(a)\}}P_{g'}^{(k')},\label{induc_-2}
		\end{align}
		and $G_g$ is given by \eqref{*} with $c_1=1,$ $b_1=q$ and
		\begin{align}
			R^1_g(y_{1:M-1};y_M,x_q-y_M,x_{1:m-1\setminus\{q,\pi(a+1)\}},a_{j_1,t_1}) =& -\frac{x_{q}\cdot a_{j_1,t_1}}{|x_{q}|^2}{x_{q}}+\hat\Lp(x_{q})R^1_{g'},\label{induc_-2.1}
	\end{align}
	where the omitted argument of $P_{g'}^{(k')}$ is
	$(y_{1:M-1};y_M,x_q-y_M,x_{1:m-1\setminus\{q,\pi(a+1)\}},\frac{j_1}{|j_1|})$ and of $R^1_{g'}$ is $(y_{1:M-1};y_M,x_q-y_M,x_{1:m-1\setminus\{q,\pi(a+1)\}},a_{j_1,t_1})$. In this case, $c_1=1,$ and $b_1=q$. 
	Since $|G_g|=|\Lp(x_{q})\tilde G_{g'}|\leq|G_{g'}|$,
	by \eqref{induc_-2}, \eqref{5.55_1} holds.
	Thus by the induction hypothesis and \eqref{induc_-1'}, we obtain \eqref{eq:bound} in this case.
	
	{\bfseries\noindent\textit{2.3.The case $c'_1b'_1=2$.} }
	The first component of $F_{g'}^N$ is $P_{g'}^{(1)}$, and the second one is $G_{g'}$. Then we have
	\begin{align}
				P_g^{(q)}(y_{1:M}; x_{1:m-1\backslash\{\pi(a+1)\}},\frac{j_1}{|j_1|}) =&\frac{1}{2\pi}\hat\Lp(x_q)P_{g'}^{(1)} ,\nonumber\\
			\bigotimes_{k\in{1:m-1\backslash\{q,\pi(a+1)\}}}P_{g}^{(k)}(y_{1:M}; x_{1:m-1\backslash\{\pi(a+1)\}},\frac{j_1}{|j_1|})=&\bigotimes_{k'\in3:m\backslash\{\pi(a)\}}P_{g'}^{(k')},\label{induc_-3}
		\end{align}
		and $G_g$ is given by \eqref{*} with $c_1=0$ and
		\begin{align}
			R^2_g(y_{1:M-1};y_M,x_q-y_M,x_{1:m-1\setminus\{q,\pi(a+1)\}},a_{j_1,t_1}) =& x_q\cdot a_{j_1,t_1}+x_q\cdot R_{g'}^1,\label{induc_-3.1}
	\end{align}
	where the omitted argument of $P_{g'}^{(1)}$ and $P_{g'}^{(k')}$ is
	$(y_{1:M-1};y_M,x_q-y_M,x_{1:m-1\setminus\{q,\pi(a+1)\}},\frac{j_1}{|j_1|})$ and of $R^1_{g'}$ is $(y_{1:M-1};y_M,x_q-y_M,x_{1:m-1\setminus\{q,\pi(a+1)\}},a_{j_1,t_1})$. 
	Since $|G_g|=|x_q\cdot\tilde G_{g'}|\leq|x_q||G_{g'}|$, 
	\eqref{5.55_1} also holds in this case.
	Thus by the induction hypothesis and \eqref{induc_-1'}, \eqref{eq:bound} follows in this case.
	
	For both $\sigma_{a+1}=\pm$, the cases we discussed above ensure that 
	if all the functions for $g'$ satisfy conditions $(i)-(v)$ so do these for $g$. 
	{For $(vi)$, $Q_g(y_{1:M};0)>0$ follows from \cite[Lemma 2.19]{CGT24}. For $P_g(l_{1:m+1\backslash\{c_1b_1,\pi(a)\}},y_{1:M};x_{1:m+1\backslash\{\pi(a)\}},\frac{j_1}{|j_1|})$, $G_g(y_{1:M};x_{1:m+1\backslash\{\pi(a)\}},a_{j_1,t_1})$ and $m\geq2$,
	when $a=1,$ this is clear by \eqref{induc}. For $a>1$ and $\sigma_a=+,$ the new appeared Leray matrices and vectors in \eqref{induc_+1}-\eqref{induc_+3.1} are different from $\hat\Lp(\sum_{i\neq\pi(a)}x_i)$ and $\sum_{i\neq\pi(a)}x_i$ since $m+1\geq 4.$ Thus the functions $P_{g}$ and $G_g$, evaluated at its new arguments, do not contain a term of the form above by the induction hypothesis.
	For $\sigma_a=-$, this is also true since in \eqref{induc_-1}-\eqref{induc_-3.1}, $m-1\geq 3.$}
\end{proof}

Now by Lemma \ref{lem5.8} we obtain the following integral representation for the limit of $\langle{\tilde\sigma}_{-\bj',\bt'},\fT^N_p[g]{\tilde\sigma}_{\bj,\bt}\rangle_{L^2_{m_\kappa}}$ according to the cases we have discussed about the operators $T^{N,+}[(i,i')]$ and $T^{N,-}[q].$ 
The key point is that the terms in this integral depending on the basis $\{a_{j_1,1},\cdots,a_{j_1,d-1},\frac{j_1}{|j_1|}\}$ are in the form of inner products, and we can use
a rotational change of coordinates to show that the integral in \eqref{liminner} is a constant independent of $j_1$.

\begin{lem}\label{lem5.9}
	Let $p\in\Pi^{(n)}_{a,1}$ and $g\in\tilde\fG^2[p]$, where $a=2b\geq 1$ is even. There exist functions $W^i_g:\RR^{b+2}\to\RR$, $i=1,2$ and $W^3_g:\RR^{b+1}\to\RR$, such that
	\begin{align}
		&\lim_{N\to \infty}\langle{\tilde\sigma}_{-\bj',\bt'},\fT^N_p[g]{\tilde\sigma}_{\bj,\bt}\rangle_{L^2_{m_\kappa}}
		=\1_{A_{\bj;\bj'}(g)}\1_{t'_{\pi(a)}=t_2}
		\frac{|j_1|^2}{|\bj|^2}
		(\lambda\iota)^a2^{bd}\times\label{liminner}\\
		&\int\mu(dy_{1:b})\frac{I_g(y_{1:b};0)}{2\pi Q_g(y_{1:b};0)}\left[W^1_g\left(y_{1:b},\frac{j_1}{|j_1|},a_{j_1,t_1}\right)W^2_g\left(y_{1:b},\frac{j_1}{|j_1|},a_{j_1,t'_{\pi(a)^c}}\right)+\1_{t'_{\pi(a)^c}=t_1}W^3_g\left(y_{1:b},\frac{j_1}{|j_1|}\right)\right]\, ,\nonumber
	\end{align}
	where $A_{\bj;\bj'}(g)=\{j'_{\pi(a)^c}=j_1;\, j'_{\pi(a)}=j_2\}$, for $\pi(a)^c=\{1,2\}\backslash\{\pi(a)\}.$ $\mu$ is the uniform probability measure on $[-1,1]^{bd}$, and $I_g, Q_g$ are the functions defined in Lemma \ref{lem5.8}.
	 The terms in $W_g^1$ (resp. $W_g^2$) are the products of inner-products between vectors in $V$ and $V'$, 
	 where $V$ is the set of vectors which are linear combinations of $y_{1:b},\frac{j_1}{|j_1|}$ and $a_{j_1,t_1}$ (resp. $a_{j_1,t'_{\pi(a)^c}}$), and $V'$ is the set of vectors in $V$ divided by its Euclidean norm.
	 The terms in $W_g^3$ have the same form depending on $y_{1:b},\frac{j_1}{|j_1|}$.
	Moreover,
	the integral in \eqref{liminner} is independent of $j_1$.
	
	If $j_1,\, j'_1\in\ZZ_0^d,\, t_1,\, t'_1\in\{1,\cdots,d-1\}$ and $g\in\fG^1[p]$, the analog of \eqref{liminner} holds upon setting $\bj=j_1,\,\bj'=j'_1$ and $\bt=t_1,\, \bt'=t'_1$, and removing the dependence of $t_2$ and $\pi(a)$. 
	Furthormore, for $g\in\fG^1[p]$ and the associated elements $\hat g$ and $\tilde g$ in $\tilde\fG^2[p]$, we have $W_g^i=W^i_{\hat g}=W^i_{\tilde g}$ for $i=1,2,3$.
\end{lem}

\begin{proof}
	We only prove \eqref{liminner} for the case $\kappa=2$. The argument for $\kappa=1$ is similar and we omit the details. 
	By \eqref{eq:kernelf} and \eqref{eq:formkernelG}, we have
	\begin{align}
		&\langle{\tilde\sigma}_{-\bj',\bt'},\fT^N_p[g]{\tilde\sigma}_{\bj,\bt}\rangle_{L^2_{m_\kappa}}\nonumber\\
		=&\sum_{l_{1:2}} a_{j'_1,t'_1}^{l_1}a_{j'_2,t'_2}^{l_2}F_g^N((l_1,N^{-1}j'_1),(l_2,N^{-1}j'_2);N^{-1}\bj,\bt)\nonumber\\
		=&\1_{A_{\bj;\bj'}(g)}\1_{t'_{\pi(a)}=t_2}\frac{|j_1|}{|\bj|^2}(\lambda\iota)^a 
		N^{1-dM}\sum_{l_{\pi(a)^c}}a_{j_1,t'_{\pi(a)^c}}^{l_{\pi(a)^c}}\sum_{\substack{y_{1:M}\in (\frac1N \ZZ^d_0)^M:\\|y_i|_\infty\le 1, 1\leq i\le M}} \frac{I_g(y_{1:M};N^{-1}j_1)}{Q_g(y_{1:M};N^{-1}\bj)}
		\times\nonumber\\
		&P_g(l_{1:2\backslash\{\pi(a),c_1b_1\}},y_{1:M};N^{-1}j_1,\frac{j_1}{|j_1|})G_g(y_{1:M};N^{-1}j_1,a_{j_1,t_1}).
		\label{eq:innerp}
	\end{align}
	Recall the definition of $G_g$ in Lemma \ref{lem5.8}. Since $m=1$, the left-most operator in the product $\fT_p^N[g]$ is $T^{N,-}[q]$, $q\in\{1,2\}$, such that $\fT^N_p[g]\tilde\sigma_{\bj,\bt} = T^{N,-}[q]\fT^N_p[g']\tilde\sigma_{\bj,\bt}.$
	From the proof of Lemma \ref{lem5.8}, there are three cases in \eqref{eq:innerp}. Now we calculate the limit of $\langle{\tilde\sigma}_{-\bj',\bt'},\fT^N_p[g]{\tilde\sigma}_{\bj,\bt}\rangle_{L^2_{m_\kappa}}$ in these three cases separately.
	
	{\bfseries\noindent\textit{1.The case $c'_1=0$ (corresponds to the case 2.1 in Lemma \ref{lem5.8}).}}
{By \eqref{eq:innerp}, in this case, $x_q = N^{-1}j'_q = N^{-1}j_1$, and $P_g = P_g^{(q)}$ in \eqref{induc_-1}. Then by \eqref{induc_-1} and \eqref{induc_-1.1}, we obtain
		\begin{align*}
		&P_g(y_{1:M};N^{-1}j_1,\frac{j_1}{|j_1|})
		=\frac{1}{2\pi}\left(\frac{j_1}{N}\cdot P_{g'}^{(2)}\right)\hat\Lp(j_1)P_{g'}^{(1)},
	\end{align*} 
	and $G_g$ is given by \eqref{*} with $c_1=0$ and
	\begin{align*}
		R^2_g(y_{1:M};N^{-1}j_1,a_{j_1,t_1}) =& R^2_{g'} (y_{1:M-1};y_M,N^{-1}j_1-y_M,a_{j_1,t_1}),
	\end{align*}
	where 
	the omitted arguments of $P_{g'}^{(k')}$ are $(y_{1:M-1};y_M,N^{-1}j_1-y_M,\frac{j_1}{|j_1|})$.}
Then \eqref{eq:innerp} equals
	\begin{align}
		&\1_{A_{\bj;\bj'}(g)}\1_{t'_{\pi(a)}=t_2}
		\frac{|j_1|^2}{|\bj|^2}{(\lambda\iota)^a}\sum_{\substack{y_{1:M}\in (\frac1N \ZZ^d_0)^M:\\|y_i|\le 1, i\le M}}\frac{N^{-dM}}{2\pi} \frac{I_g(y_{1:M};N^{-1}j_1)}{Q_g(y_{1:M};N^{-1}\bj)}
		R_{g'}^2(y_{1:M-1};y_M,N^{-1}j_1-y_M,a_{j_1,t_1})
		\times\nonumber\\
		&\left(\frac{j_1}{|j_1|}\cdot P_{g'}^{(2)}(y_{1:M-1};y_M,N^{-1}j_1-y_M,\frac{j_1}{|j_1|})\right)\left(a_{j_1,t'_{\pi(a)^c}}\cdot P_{g'}^{(1)}(y_{1:M-1};y_M,N^{-1}j_1-y_M,\frac{j_1}{|j_1|})\right).\label{case1}
	\end{align}
{Recall from \cite[Lemma 2.20]{CGT24}, since we assume $N\in\NN+\frac12,$ $[-1,1]^d$ can be written as the union of cubes $C(y^{*})$ of side $N^{-1}$, centered at $y^{*}\in N^{-1}\ZZ^d\cap [-1,1]^d.$ For every $y\in [-1,1]^d,$ denote $y^{*}(y)$ as the point $y^{*}$ such that $y\in C(y^{*})$.
	Then the sum in \eqref{case1} can be written as 
	\begin{align}
		&2^{bd}\int\mu(dy_{1:b})\frac{1}{2\pi}\frac{I_g^{(N)}(y_{1:b};N^{-1}j_1)}{Q_g^{(N)}(y_{1:b};N^{-1}\bj)}R_{g'}^{2,(N)}(y_{1:b-1};y_b,N^{-1}j_1-y_b,a_{j_1,t_1})
		\times\nonumber\\
		&\left(\frac{j_1}{|j_1|}\cdot P_{g'}^{(2),(N)}(y_{1:b-1};y_b,N^{-1}j_1-y_b,\frac{j_1}{|j_1|})\right)\left(a_{j_1,t'_{\pi(a)^c}}\cdot P_{g'}^{(1),(N)}(y_{1:b-1};y_b,N^{-1}j_1-y_b,\frac{j_1}{|j_1|})\right),\label{5.41}
	\end{align}
		where $\mu$ is the uniform probability measure on $[-1,1]^{bd}$, and
		\begin{align*}
			Q_g^{(N)}(y_{1:b};N^{-1}\bj):=Q_g(y^{*}(y_1),\cdots,y^{*}(y_b);N^{-1}\bj)\1_{y_i\notin[-1/2N^{-1},1/2N^{-1}]^d,i\leq b}\,,
		\end{align*} and similarly for the other functions.
	Moreover, by \eqref{eq:formkernelG} and \eqref{induc_-1'}, we have}
	\begin{align}
		&\langle{\tilde\sigma}_{-\bj',\bt'}, {\fT^{N,\delta}_p[g]}{\tilde\sigma}_{\bj,\bt}\rangle_{L^2(\TT^d,\RR^d)^{\otimes m_\kappa}}\nonumber\\
		=&
		\frac{1}{\sqrt{d-1}}\1_{A_{\bj;\bj'}(g)}\1_{t'_{\pi(a)}=t_2}
		\frac{\lambda^a |j_1|^{2\delta}}{|\bj|^{2\delta}}2^{bd}\int\mu(dy_{1:b})\left|\frac{P^{1,(N)}_{g'}(y_{1:b-1};y_b,N^{-1}j_1-y_b)}{2\pi Q_g^{(N)}(y_{1:b};N^{-1}\bj)}\right|^\delta I_g^{(N)}(y_{1:b};N^{-1}j_1).\label{5.44}
	\end{align} 
{Since by \eqref{e:delta}, for $1\leq\delta\leq d/2,$ $\langle{\tilde\sigma}_{-\bj',\bt'}, {\fT^{N,\delta}_p[g]}{\tilde\sigma}_{\bj,\bt}\rangle_{L^2(\TT^d,\RR^d)^{\otimes m_\kappa}}$ is uniformly bounded in $N$, the integrand in \eqref{5.44} is uniform integrable for $\delta=1$.
By \eqref{eq:bound}, the absolute value of the integrand in \eqref{5.41} 
is bounded above by that in \eqref{5.44}.
Thus the integrand in \eqref{5.41} is also uniform integrable, and the limit of the integral \eqref{5.41} exists and equals 
\begin{align}\label{5.45}
	&\frac{2^{bd}}{2\pi}\int\mu(dy_{1:b})\frac{I_g(y_{1:b};0)}{Q_g(y_{1:b};0)}R^2_{g'}(y_{1:b-1};y_b,-y_b,a_{j_1,t_1})\times\nonumber\\
	&\left(\frac{j_1}{|j_1|}\cdot P_{g'}^{(2)}(y_{1:b-1};y_b,-y_b,\frac{j_1}{|j_1|})\right)
	\left(a_{j_1,t'_{\pi(a)^c}}\cdot P_{g'}^{(1)}(y_{1:b-1};y_b,-y_b,\frac{j_1}{|j_1|})\right).
\end{align}
Then \eqref{liminner} holds with
\begin{align}
	W_g^1(y_{1:b},\frac{j_1}{|j_1|},a_{j_1,t_1})=&R_{g'}^2(y_{1:b-1};y_b,-y_b,a_{j_1,t_1})\left(\frac{j_1}{|j_1|}\cdot P_{g'}^{(2)}(y_{1:b-1};y_b,-y_b,\frac{j_1}{|j_1|})\right),\nonumber\\
	W_g^2(y_{1:b},\frac{j_1}{|j_1|},a_{j_1,t'_{\pi(a)^c}})=&a_{j_1,t'_{\pi(a)^c}}\cdot P_{g'}^{(1)}(y_{1:b-1};y_b,-y_b,\frac{j_1}{|j_1|}),\label{W1}
\end{align}
and $W_g^3 = 0$. By condition (v) and (vi) in Lemma \ref{lem5.8}, 
the terms in $W_g^1$ (resp. $W_g^2$) are the products of inner-products between vectors in $V$ and $V'$.
}
	
		{\bfseries\noindent\textit{2.The case $c_1b_1=\pi(a)^c$ (corresponds to the case 2.2 in Lemma \ref{lem5.8}).}} 
{By \eqref{eq:innerp}, in this case, $x_q = N^{-1}j'_q = N^{-1}j_1$ in \eqref{induc_-2}. Then by \eqref{induc_-2} and \eqref{induc_-2.1}, we obtain
		\begin{align*}
			P_g(y_{1:M};N^{-1}j_1,\frac{j_1}{|j_1|})&=\frac{1}{2\pi}\frac{j_1}{N}\cdot P_{g'}^{(2)}(y_{1:M-1};y_M,N^{-1}j_1-y_M,\frac{j_1}{|j_1|}),
		\end{align*}
	and $G_g$ is given by \eqref{*} with $c_1b_1=\pi(a)^c$ and
	\begin{align*}
		R_g^1(y_{1:M};N^{-1}j_1,a_{j_1,t_1})
		&=\hat\Lp(j_1) R_{g'}^1(y_{1:M-1};y_M,N^{-1}j_1-y_M,a_{j_1,t_1}).
	\end{align*}
	Then \eqref{eq:innerp} equals
	\begin{align}
		&\1_{A_{\bj;\bj'}(g)}\1_{t'_{\pi(a)}=t_2}
		\frac{|j_1|^2}{|\bj|^2}(\lambda\iota)^a
		\sum_{\substack{y_{1:M}\in (\frac1N \ZZ^d_0)^M:\\|y_i|_\infty\le 1, i\le M}} \frac{N^{-dM}I_g(y_{1:M};N^{-1}j_1)}{2\pi Q_g(y_{1:M};N^{-1}\bj)}
		\times\nonumber\\
		&\left(\frac{j_1}{|j_1|}\cdot P_{g'}^{(2)}(y_{1:M-1};y_M,N^{-1}j_1-y_M,\frac{j_1}{|j_1|})\right)\times\nonumber\\
		&\left(\1_{t_1=t'_{\pi(a)^c}}+a_{j_1,t'_{\pi(a)^c}}\cdot R_{g'}^1(y_{1:M-1};y_M,N^{-1}j_1-y_M,a_{j_1,t_1})\right).\label{5.58}
	\end{align}
Similar as the case 1, the limit of the summation in \eqref{5.58} exists, and equals 
\begin{align}
	&2^{bd}\int\mu(dy_{1:b})\frac{1}{2\pi}\frac{I_g(y_{1:b};0)}{Q_g(y_{1:b};0)}
	\left(\frac{j_1}{|j_1|}\cdot P_{g'}^{(2)}(y_{1:b-1};y_b,-y_b,\frac{j_1}{|j_1|})\right)\times\nonumber\\
	&\left(\1_{t_1=t'_{\pi(a)^c}}+a_{j_1,t'_{\pi(a)^c}}\cdot R_{g'}^{1}(y_{1:b-1};y_b,-y_b,a_{j_1,t_1})\right).\label{limC3.1}
\end{align}
By condition (v) in Lemma \ref{lem5.8}, the terms involving $a_{j_1,t_1}$ in $R_{g'}^{1}$ are in the form of the inner-products with the vectors in $V$ and $V'$. Thus we can write $R_{g'}^{1}(y_{1:b-1};y_b,-y_b,a_{j_1,t_1})$ $=R_{g'}^{1,1}(y_{1:b},a_{j_1,t_1})R_{g'}^{1,2}(y_{1:b})$ with the real-valued function $R_{g'}^{1,1}:\RR^{b+1}\to\RR$ and vector-valued function $R_{g'}^{1,2}:\RR^{b+1}\to\RR^d.$
Thus \eqref{liminner} holds with
\begin{align}
	W_g^{1}(y_{1:b},\frac{j_1}{|j_1|},a_{j_1,t_1})=&
	\left(\frac{j_1}{|j_1|}\cdot P_{g'}^{(2)}(y_{1:b-1};y_b,-y_b,\frac{j_1}{|j_1|})\right) R_{g'}^{1,1}(y_{1:b},a_{j_1,t_1}),\nonumber\\
 W_g^{2}(y_{1:b},\frac{j_1}{|j_1|},a_{j_1,t'_{\pi(a)^c}})=&a_{j_1,t'_{\pi(a)^c}}\cdot R_{g'}^{1,2}(y_{1:b}),\label{W2}
\end{align}
and
\begin{align*}
 W_g^{3}(y_{1:b},\frac{j_1}{|j_1|})=&\frac{I_g(y_{1:b};0)}{2\pi Q_g(y_{1:b};0)}
	\left(\frac{j_1}{|j_1|}\cdot P_{g'}^{(2)}(y_{1:b-1};y_b,-y_b,\frac{j_1}{|j_1|})\right).
\end{align*}
By condition (v) in Lemma \ref{lem5.8}, $W_g^i,i=1,2,3$ satisfy the condtion in Lemma \ref{lem5.9}.
}

	{\bfseries\noindent\textit{3.The case $c'_1b'_1=2$ (corresponds to the case 2.3 in Lemma \ref{lem5.8}).}}
{By \eqref{eq:innerp}, in this case, $x_q = N^{-1}j'_q = N^{-1}j_1$, and $P_g = P_g^{(q)}$ in \eqref{induc_-3}. Then by \eqref{induc_-3} and \eqref{induc_-3.1}, we obtain
	\begin{align*}
		P_g(y_{1:M};N^{-1}j_1,\frac{j_1}{|j_1|})&=\frac{1}{2\pi}\hat\Lp(j_1) P_{g'}^{(1)}(y_{1:M-1};y_M,N^{-1}j_1-y_M,\frac{j_1}{|j_1|}),
	\end{align*}
	and $G_g$ is given by \eqref{*} with $c_1=0$ and
	\begin{align*}
		R_g^2(y_{1:M};N^{-1}j_1,a_{j_1,t_1})
		=&\frac{j_1}{N}\cdot R_{g'}^1(y_{1:M-1};y_M,N^{-1}j_1-y_M,a_{j_1,t_1}).
	\end{align*}
	Then \eqref{eq:innerp} in this case equals
	\begin{align}
		&\1_{A_{\bj;\bj'}(g)}\1_{t'_{\pi(a)}=t_2}
		\frac{|j_1|^2}{|\bj|^2}(\lambda\iota)^a 
		\sum_{\substack{y_{1:M}\in (\frac1N \ZZ^d_0)^M:\\|y_i|_\infty\le 1, i\le M}} \frac{N^{-dM}I_g(y_{1:M};N^{-1}j_1)}{2\pi Q_g(y_{1:M};N^{-1}\bj)}
		\times\nonumber\\
		&\left(a_{j_1,t'_{\pi(a)^c}}\cdot P_{g'}^{(1)}(y_{1:M-1};y_M,N^{-1}j_1-y_M,\frac{j_1}{|j_1|})\right)\left(\frac{j_1}{|j_1|}\cdot R_{g'}^1(y_{1:M-1};y_M,N^{-1}j_1-y_M,a_{j_1,t_1})\right).\label{case2}
	\end{align}
	Similar as case 1, the limit of the summation in \eqref{case2} exists, and equals 
\begin{align*}
	&2^{bd}\int\mu(dy_{1:b})\frac{1}{2\pi}\frac{I_g(y_{1:b};0)}{Q_g(y_{1:b};0)}
	\left(\frac{j_1}{|j_1|}\cdot R_{g'}^{1}(y_{1:b-1};y_b,-y_b,a_{j_1,t_1})\right)\left(a_{j_1,t'_{\pi(a)^c}}\cdot P_{g'}^{(1)}(y_{1:b-1};y_b,-y_b,\frac{j_1}{|j_1|})\right).
\end{align*}
Thus \eqref{liminner} holds with
\begin{align}
	W_g^{1}(y_{1:b},\frac{j_1}{|j_1|},a_{j_1,t_1})=&
\frac{j_1}{|j_1|}\cdot R_{g'}^{1}(y_{1:b-1};y_b,-y_b,a_{j_1,t_1}),\nonumber\\
	W_g^{2}(y_{1:b},\frac{j_1}{|j_1|},a_{j_1,t'_{\pi(a)^c}})=&a_{j_1,t'_{\pi(a)^c}}\cdot P_{g'}^{(1)}(y_{1:b-1};y_b,-y_b,\frac{j_1}{|j_1|}),\label{W3}
\end{align}
and $W_g^3=0.$
By condition (v) in Lemma \ref{lem5.8}, 
$W_g^i,i=1,2,3$ satisfy the condtion in Lemma \ref{lem5.9}. 
}

Moreover, since all the functions in the above three cases are independent of $j'_{\pi(a)}=j_2$,
we have $W^i_g=W^i_{\hat g}=W^i_{\tilde g}$ for $i=1,2,3.$
{
Now we prove that the integral in \eqref{liminner} is independent of $j_1$.
For another $\tilde j_1\in\ZZ_0^d$, there exists an unitary matrix $U$ such that $\frac{j_1}{|j_1|}=U\frac{\tilde j_1}{|\tilde j_1|}$.
Recall that for $j_1\in\ZZ_0^d$, $\{a_{j_1,1},\cdots, a_{j_1,d-1},\frac{j_1}{|j_1|}\}$ is a right-handed othonormal basis in $\RR^d$. 
Then $a_{j_1,t_1} = Ua_{\tilde j_1,t_1}$ and $a_{j_1,t'_{\pi(a)^c}} = Ua_{\tilde j_1,t'_{\pi(a)^c}}$.
Since the terms in $W_g^i$, $i=1,2,3$ are inner products of vectors in $V$ and $V'$, combining conditions (i),(ii) in Lemma \ref{lem5.8}, for $y_{1:b}\in[-1,1]^{bd}$, we have
\begin{align*}
	&\frac{I_g(y_{1:b};0)}{2\pi Q_g(y_{1:b};0)}\left[W^1_g\left(y_{1:b},\frac{j_1}{|j_1|},a_{j_1,t_1}\right)W^2_g\left(y_{1:b},\frac{j_1}{|j_1|},a_{j_1,t'_{\pi(a)^c}}\right)+\1_{t'_{\pi(a)^c}=t_1}W^3_g\left(y_{1:b},\frac{j_1}{|j_1|}\right)\right]\\
	=&\frac{I_g(U^Ty_{1:b};0)}{2\pi Q_g(U^Ty_{1:b};0)}\left[W^1_g\left(U^Ty_{1:b},\frac{\tilde j_1}{|\tilde j_1|},a_{\tilde j_1,t_1}\right)W^2_g\left(U^Ty_{1:b},\frac{\tilde j_1}{|\tilde j_1|},a_{\tilde j_1,t'_{\pi(a)^c}}\right)+\1_{t'_{\pi(a)^c}=t_1}W^3_g\left(U^Ty_{1:b},\frac{\tilde j_1}{|\tilde j_1|}\right)\right],
\end{align*}
where $U^T$ is the transpose of $U$, and $U^Ty_{1:b}:=(U^Ty_1,\cdots,U^Ty_b)$.
Since the Lebesgue measure is invariant under unitary transformations, performing the change of variables $z_{1:b}=U^Ty_{1:b}$ yields
\begin{align*}
	&\int\mu(\ud y_{1:b})\frac{I_g(y_{1:b};0)}{2\pi Q_g(y_{1:b};0)}\left[W^1_g\left(y_{1:b},\frac{j_1}{|j_1|},a_{j_1,t_1}\right)W^2_g\left(y_{1:b},\frac{j_1}{|j_1|},a_{j_1,t'_{\pi(a)^c}}\right)+\1_{t'_{\pi(a)^c}=t_1}W^3_g\left(y_{1:b},\frac{j_1}{|j_1|}\right)\right]\\
	=&\int\mu(\ud z_{1:b})\frac{I_g(z_{1:b};0)}{2\pi Q_g(z_{1:b};0)}\left[W^1_g\left(z_{1:b},\frac{\tilde j_1}{|\tilde j_1|},a_{\tilde j_1,t_1}\right)W^2_g\left(z_{1:b},\frac{\tilde j_1}{|\tilde j_1|},a_{\tilde j_1,t'_{\pi(a)^c}}\right)+\1_{t'_{\pi(a)^c}=t_1}W^3_g\left(z_{1:b},\frac{\tilde j_1}{|\tilde j_1|}\right)\right],
\end{align*}
which impies that the integral in \eqref{liminner} is independent of $j_1$.
}
\end{proof}
{
Now we are ready to prove Proposition 
\ref{prop:limite}. The key point is to prove the sum of the integrals in \eqref{liminner} over $g\in\fG^1[p]$ equals $\1_{t'_1=t_1}$ times a constant when $\kappa=1$, which implies the limit equation is the decoupled stochastic heat equation. This is not obvious and in the following, we will use the Hermitian property of the operator $\sum_{p\in\Pi_{a,1}^{(n)}}\fT_p^N$ and a rotational change of coordinates in the integral.
}
\begin{proof}[Proof of Proposition \ref{prop:limite}]
	We begin with $\kappa=1$. In this case, $A_{\bj;\bj'}(g)=\{j'_1=j_1\}$, $\tilde\sigma_{\bj,\bt}=\sigma_{\bj,\bt}$ in \eqref{liminner}.
	Thus by Lemma \ref{lem5.9}, we have
	\begin{align}
		&\lim_{N\to\infty}\sum_{p\in \Pi^{(n)}_{a,1}}\sum_{g\in\fG^1[p]}\langle{\sigma}_{-\bj',\bt'},\fT^N_p[g]{\sigma}_{\bj,\bt}\rangle_{L^2_1}\nonumber\\
		=&\1_{j'_1=j_1}
		(\lambda\iota)^a2^{bd}\sum_{p\in \Pi^{(n)}_{a,1}}\sum_{g\in\fG^1[p]}\int\mu(dy_{1:b})\frac{I_g(y_{1:b};0)}{2\pi Q_g(y_{1:b};0)}\nonumber\\
		&\left[W^1_g\left(y_{1:b},\frac{j_1}{|j_1|},a_{j_1,t_1}\right)W^2_g\left(y_{1:b},\frac{j_1}{|j_1|},a_{j_1,t'_1}\right)+\1_{t'_1=t_1}W^3_g\left(y_{1:b},\frac{j_1}{|j_1|}\right)\right]\nonumber\\
		=:&\1_{j'_1=j_1}B(t'_1,t_1).\label{lim}
	\end{align}
{For $p\in\Pi_{a,1}^{(n)}$, since $a$ is even, $\langle\sigma_{-\bj',\bt'},\fT^N_p\sigma_{\bj,\bt}\rangle_{L^2_1}$ is real-valued. 
Set $f(p)=(p_a,p_{a-1},\cdots,p_0)$ be a path in $\Pi_{a,1}^{(n)}$. Then $f$ is a bijection on $\Pi_{a,1}^{(n)}$,
and by $(T^{N,+})^{*}=-T^{N,-}$, we have $(\fT_p^N)^*=\fT_{f(p)}^N$. Therefore there holds
	\begin{align}
		\sum_{p\in\Pi_{a,1}^{(n)}}\langle\sigma_{-\bj',\bt'},\fT^N_p\sigma_{\bj,\bt}\rangle_{L^2_1}
		=\sum_{p\in\Pi_{a,1}^{(n)}}\langle\sigma_{\bj,\bt},\fT^N_{f(p)}\sigma_{-\bj',\bt'}\rangle_{L_1^2}
		=\sum_{p\in\Pi_{a,1}^{(n)}}\langle\sigma_{\bj,\bt},\fT^N_{p}\sigma_{-\bj',\bt'}\rangle_{L_1^2}
		\,.\label{5.64}
	\end{align}
}{Since $B(t'_1,t_1)$ is independent of $j_1$ by Lemma \ref{lem5.9}, by \eqref{lim} and \eqref{5.64}, we derive that
\begin{align}\label{B}
	B(t'_1,t_1)=B(t_1,t'_1),\quad t_1,\, t'_1\in\{1,\cdots,d-1\}.
\end{align} 
Next we prove that it equals $\1_{t_1=t'_1}B(1,1)$.
Let $1\leq i<r\leq d-1$ be fixed. 
Recall that $\{a_{j_1,1},\cdots, a_{j_1,d-1},\frac{j_1}{|j_1|}\}$ is a right-handed othonormal basis in $\RR^d$.
Now we change the order of $a_{j_1,i}$ and $a_{j_1,r}$, and also the sign of $a_{j_1,i}$. Then
$\{a_{j_1,1},\cdots,a_{j_1,i-1},a_{j_1,r},a_{j_1,i+1},\cdots,a_{j_1,r-1},-a_{j_1,i},a_{j_1,r+1},$ $\cdots,a_{j_1,d-1},\frac{j_1}{|j_1|}\}$ still form a right-handed othonormal basis.
Thus there exists an unitary matrix $U$ such that $Ua_{j_1,i} = a_{j_1,r}$, $Ua_{j_1,r}=-a_{j_1,i}$ and $U\frac{j_1}{|j_1|}=\frac{j_1}{|j_1|}.$
{By Lemma \ref{lem5.9}, the terms in $W_g^i$, $i=1,2,3$ are inner products of vectors in $V$ and $V'$, and by \eqref{W1}, \eqref{W2} and \eqref{W3}, $W^2_g\left(y_{1:b},\frac{j_1}{|j_1|},a_{j_1,r}\right)$ is linear in $a_{j_1,r}$.}
Then for each $g\in\fG^1[p]$, 
we have \begin{align*}
	W^1_g\left(y_{1:b},\frac{j_1}{|j_1|},a_{j_1,i}\right)W^2_g\left(y_{1:b},\frac{j_1}{|j_1|},a_{j_1,r}\right)
	&=	W^1_g\left(Uy_{1:b},\frac{j_1}{|j_1|},a_{j_1,r}\right)W^2_g\left(Uy_{1:b},\frac{j_1}{|j_1|},-a_{j_1,i}\right)\\
	&=-W^1_g\left(Uy_{1:b},\frac{j_1}{|j_1|},a_{j_1,r}\right)W^2_g\left(Uy_{1:b},\frac{j_1}{|j_1|},a_{j_1,i}\right)\,.
\end{align*}
Moreover, by consition (i),(ii) in Lemma \ref{lem5.8}, $I_g(y_{1:b};0)=I_g(Uy_{1:b};0)$ and $Q_g(y_{1:b};0)=Q_g(Uy_{1:b};0)$. 
Since the Lebesgue measure is invariant under unitary transformations, by performing the change of variables $z_{1:b}=Uy_{1:b}$, we have 
\begin{align*}
&\int\mu(dy_{1:b})\frac{I_g(y_{1:b};0)}{2\pi Q_g(y_{1:b};0)}	W^1_g\left(y_{1:b},\frac{j_1}{|j_1|},a_{j_1,i}\right)W^2_g\left(y_{1:b},\frac{j_1}{|j_1|},a_{j_1,r}\right)\\
	=&-\int\mu(dz_{1:b})\frac{I_g(z_{1:b};0)}{2\pi Q_g(z_{1:b};0)}	W^1_g\left(z_{1:b},\frac{j_1}{|j_1|},a_{j_1,r}\right)W^2_g\left(z_{1:b},\frac{j_1}{|j_1|},a_{j_1,i}\right),
\end{align*}
which implies $B(i,r)=-B(r,i)$.
Therefore, by \eqref{B} for $i\neq r$, $B(i,r)=0$.
Let $1< i\leq d-1$ be fixed. 
We change the order of $a_{j_1,1}$ and $a_{j_1,i}$, and also the sign of $a_{j_1,i}$. Then
$\{-a_{j_1,i},\cdots,a_{j_1,i-1},a_{j_1,1},$ $a_{j_1,i+1},\cdots,a_{j_1,d-1},\frac{j_1}{|j_1|}\}$ still form a right-handed othonormal basis, and there exists an unitary matrix $U$ such that $Ua_{j_1,1} = -a_{j_1,i}$ and $U\frac{j_1}{|j_1|}=\frac{j_1}{|j_1|}.$
Then for each $g\in\fG^1[p]$, 
we have \begin{align*}
	W^1_g\left(y_{1:b},\frac{j_1}{|j_1|},a_{j_1,1}\right)W^2_g\left(y_{1:b},\frac{j_1}{|j_1|},a_{j_1,1}\right)
	&=	W^1_g\left(Uy_{1:b},\frac{j_1}{|j_1|},-a_{j_1,i}\right)W^2_g\left(Uy_{1:b},\frac{j_1}{|j_1|},-a_{j_1,i}\right)\\
	&=W^1_g\left(Uy_{1:b},\frac{j_1}{|j_1|},a_{j_1,i}\right)W^2_g\left(Uy_{1:b},\frac{j_1}{|j_1|},a_{j_1,i}\right)\,.
\end{align*}
Then similarly to the above, by performing the change of variables $z_{1:b}=U^Ty_{1:b}$, we have
$B(i,i)=B(1,1).$ 
Therefore, $B(i,j)=\1_{i=j}B(1,1)=:\1_{i=j}c(a,\lambda)$. Hence \eqref{eq:liminner1} follows.
}

For $\kappa=2,$ recall from subsection \ref{sec4.1} that there exists a correspondence between $\fG^1[p]$ and $\tilde\fG^2[p]$.
For every $g\in\fG^1[p]$, set $\hat g$ and $\tilde g$ be the associated elements in $\tilde\fG^2[p]$. Then we have
	\begin{align}
		&\lim_{N\to \infty}\sum_{p\in \Pi^{(n)}_{a,1}}\sum_{g\in\tilde\fG^2[p]}\langle{\tilde\sigma}_{-\bj',\bt'},\fT^N_p[g]{\tilde\sigma}_{\bj,\bt}\rangle_{L^2_{m_\kappa}}\nonumber\\
		=&\lim_{N\to \infty}\sum_{p\in \Pi^{(n)}_{a,1}}\sum_{g\in\fG^1[p]}\left(\langle{\tilde\sigma}_{-\bj',\bt'},\fT^N_p[\hat g]{\tilde\sigma}_{\bj,\bt}\rangle_{L^2_{m_\kappa}}+\langle{\tilde\sigma}_{-\bj',\bt'},\fT^N_p[\tilde g]{\tilde\sigma}_{\bj,\bt}\rangle_{L^2_{m_\kappa}}\right).\label{5.75}
	\end{align}
	Without loss of generality, we set $\pi(a)$ in $\hat g$ be 1, while in $\tilde g$ be 2. By Lemma \ref{lem5.9}, for each $g\in\tilde\fG^2[p]$ and $p\in\Pi_{a,1}^{(n)}$, we have $W_g^i=W^i_{\hat g}=W^i_{\tilde g}$ for $i=1,2,3.$ Thus by \eqref{liminner} and \eqref{lim}, we obtain \eqref{5.75} equals
	\begin{align}
		&\frac{|j_1|^2}{|\bj|^2}\left(\1_{j'_2=j_1,j'_1=j_2}\1_{t'_1=t_2}B(t'_2,t_1)+\1_{j'_1=j_1,j'_2=j_2}\1_{t'_2=t_2}B(t'_1,t_1)\right).\label{lim1}
	\end{align}
	Since $B(t'_2,t_1)=\1_{t'_2=t_1}c(a,\lambda),\, B(t'_1,t_1)=\1_{t'_1=t_1}c(a,\lambda)$, \eqref{eq:liminner2} follows.
\end{proof}

\section{Asymptotic expansion of the effective coefficient when $d\geq3$}\label{sec5}
Let $k\in\ZZ_0^d$ be fixed throughout the section.
Recall the definition of $T^{N,\pm}$ and $T^{N}_{i,n}$ in \eqref{4.2}. In this section, we prove an asymptotic expansion w.r.t $\lambda$ of the effective coefficient when $d\geq3$. To this end, we introduce the following operators which are independent of $\lambda.$

Let $T^{N,\pm,*} := \lambda^{-1} T^{N,\pm},\, T^{N,*}_{2,n} := \lambda^{-1}T^{N}_{2,n}$. 
First, we prove that
\begin{align}\label{**}
	D = \sum_{l=1}^mf_l\lambda^{2l}+R_m,
\end{align}
where
\begin{align}
	f_l =& (-1)^{l-1}\lim_{n\to\infty}\lim_{N\to\infty} \|(T^{N,*}_{2,n})^{l-1}T^{N,+,*}\sigma_{k,1}\|^2\,,\nonumber\\
	R_m =&(-1)^{m}\lim_{n\to\infty}\lim_{N\to\infty}\langle(-\fL_0)^{\frac12}(T^{N}_{2,n})^{m}T^{N,+}\sigma_{k,1},(-\fL^N_{2,n})^{-1}(-\fL_0)^{\frac12}(T^{N}_{2,n})^{m}T^{N,+}\sigma_{k,1}\rangle\,,\label{eq:expansion}
\end{align}
{and there exists a positive constant $C$, depending only on $d$, such that 
	\begin{align*}
		|f_l|\leq l!C^l,\quad |R_m|\leq(m+1)!C^{m+1}\lambda^{2m+2}.
	\end{align*}}
\begin{proof}[Proof of \eqref{**}]
	We proceed by induction. 
	{For $m=1$, let $v^{N,n}[-k]$ be the solution to \eqref{eq:geneq} with $f_1=\sigma_{k,1}$.
	Then by \eqref{5.3} and \eqref{C.10}, we have
	\begin{align*}
		D=&\lim_{n\to\infty}\lim_{N\to\infty}\frac{1}{2\pi|k|}\|(-\fL_0)^{-\frac12}\fA^N_-v_2^{N,n}[-k]\|
		=\lim_{n\to\infty}\lim_{N\to\infty}\frac{1}{(2\pi|k|)^2}\langle{\fA_+^N\sigma_{k,1}},(-\fL^N_{2,n})^{-1}\fA_+^N\sigma_{k,1}\rangle\nonumber\\
		=&\lim_{n\to\infty}\lim_{N\to\infty}\langle{(-\fL_0)^{\frac12}T^{N,+}\sigma_{k,1}},(-\fL^N_{2,n})^{-1}(-\fL_0)^{\frac12}T^{N,+}\sigma_{k,1}\rangle\,.
	\end{align*} 
	Recall that the functions in the Fock space $\FK$ are mean-zero, which results that $-\fL_0\geq (2\pi)^2.$
	Then by the variational formula \cite[Theorem 4.1]{KLO12} with $\lambda=1$ and $L=(\fL_0+1)+\fA_{2,n}^N$, we have}
	\begin{align*}
		D=&\lim_{n\to\infty}\lim_{N\to\infty}\inf_{\rho\in \FK}\Big\{\|(-\fL_0)^{\frac12}\rho\|^2+\|T^{N,+}\sigma_{k,1}+(-\fL_0)^{-\frac12}\fA^N_{2,n}\rho\|^2\Big\}\nonumber\\
		=&\lim_{N\to\infty}\|T^{N,+}\sigma_{k,1}\|^2-\lim_{n\to\infty}\lim_{N\to\infty}\sup_{\rho\in \FK}\Big\{2\langle(-\fL_0)^{\frac12}T^{N}_{2.n}T^{N,+}\sigma_{k,1},\rho\rangle-\|(-\fL_0)^{\frac12}\rho\|^2-\|(-\fL_0)^{-\frac12}\fA^N_{2,n}\rho\|^2\Big\}\,.
	\end{align*}
Then applying variational formula \cite[Theorem 4.1]{KLO12} again, the above equals
\begin{align*}
	\lambda^2\lim_{N\to\infty}\|T^{N,+,*}\sigma_{k,1}\|^2-\lim_{n\to\infty}\lim_{N\to\infty}\langle(-\fL_0)^{\frac12}T^{N}_{2,n}T^{N,+}\sigma_{k,1},(-\fL^N_{2,n})^{-1}(-\fL_0)^{\frac12}T^{N}_{2,n}T^{N,+}\sigma_{k,1}\rangle\,.
\end{align*}
Thus \eqref{eq:expansion} holds for $m=1$.

Assume now \eqref{eq:expansion} holds for some $m \geq 1$, we prove it for $m+1$.
Applying the variational formula once more to
$\langle(-\fL_0)^{\frac12}(T^{N}_{2,n})^{m}T^{N,+}\sigma_{k,1},(-\fL^N_{2,n})^{-1}(-\fL_0)^{\frac12}(T^{N}_{2,n})^{m}T^{N,+}\sigma_{k,1}\rangle$ in the remainder term, we obtain
	\begin{align*}
		&\inf_{\rho\in \FK}\Big\{\|(-\fL_0)^{\frac12}\rho\|^2+\|(T^{N}_{2,n})^{m}T^{N,+}\sigma_{k,1}+(-\fL_0)^{-\frac12}\fA^N_{2,n}\rho\|^2\Big\}\\
		=&\|(T^{N}_{2,n})^mT^{N,+}\sigma_{k,1}\|^2-\sup_{\rho\in \FK}\Big\{2\langle(-\fL_0)^{\frac12}(T^{N}_{2,n})^{m+1}T^{N,+}\sigma_{k,1},\rho\rangle-\|(-\fL_0)^{\frac12}\rho\|^2-\|(-\fL_0)^{-\frac12}\fA^N_{2,n}\rho\|^2\Big\}\\
		=&\lambda^{2m+2}\|(T^{N,*}_{2,n})^mT^{N,+,*}\sigma_{k,1}\|^2-\langle(-\fL_0)^{\frac12}(T^{N}_{2,n})^{m+1}T^{N,+}\sigma_{k,1},(-\fL^N_{2,n})^{-1}(-\fL_0)^{\frac12}(T^{N}_{2,n})^{m+1}T^{N,+}\sigma_{k,1}\rangle\,.
	\end{align*}
Substitute the above into $R_m$, we derive \eqref{eq:expansion} for $m+1$.
	
	{
	Finally, we estimate $f_l$ and the $\lambda$-dependence of $R_m$ in \eqref{eq:expansion}.
	By \eqref{eq:sector}, there exists a positive constant $C$, depending only on $d$, such that for any $n\geq 2$ and $N\geq1$,
	$$\|(T^{N,*}_{2,n})^{l-1}T^{N,+,*}\sigma_{k,1}\|^2\leq l!C^l \,,$$
	which implies $|f_l|\leq l!C^l.$
	For the remainder $R_m$,
	by the variational formula \cite[Theorem 4.1]{KLO12} and \eqref{eq:sector}, we have
	\begin{align*}
		&\langle(-\fL_0)^{\frac12}(T^{N}_{2,n})^{m}T^{N,+}\sigma_{k,1},(-\fL^N_{2,n})^{-1}(-\fL_0)^{\frac12}(T^{N}_{2,n})^{m}T^{N,+}\sigma_{k,1}\rangle\\
		\leq& \|(T^{N}_{2,n})^{m}T^{N,+}\sigma_{k,1}\|^2\leq (m+1)!C^{m+1}\lambda^{2m+2} \,,
	\end{align*}
	which implies that $|R_m|\leq(m+1)!C^{m+1}\lambda^{2m+2}$.
	Hence the results follows.
}
\end{proof}

Next, we prove Corollary \ref{cor}.
As shown above, Theorem~\ref{thm:main} provides an asymptotic expansion of $D$ in terms of powers of $\lambda$.  
We now show that $D$ differs from both the conjectured constant $\nu_{\mathrm{eff}}-1$ in \cite[Conjecture~6.5]{JP24} and the constant $D_{\mathrm{rep}}$ defined in Remark~\ref{rmk}.

\begin{proof}[Proof of Corollary \ref{cor}]
	(i){ By \eqref{3.9}, \eqref{3.10} and \eqref{**} for $m=1$, we have
	\begin{align}
		D\leq& \lim_{N\to\infty}\|T^{N,+}\sigma_{k,1}\|^2\nonumber\\
		=&\lim_{N\to\infty}\frac{2!\lambda_N^2}{(2\pi)^2|k|^2}\sum_{l_{1:2}}\sum_{k_{1:2}}\frac{(\fR_{k_1,k_2}^N)^2}{4|k_{1:2}|^2}\1_{k_1+k_2=k}
		\left(\left(\hat\Lp(k_1)(k_1+k_2)\right)(l_1)\left(\hat\Lp(k_2)a_{k,1}\right)(l_2)\right.\nonumber\\
		&\left.+\left(\hat\Lp(k_2)(k_1+k_2)\right)(l_2)\left(\hat\Lp(k_1)a_{k,1}\right)(l_1)\right)^2\nonumber\\
		=&\lim_{N\to \infty}\frac{\lambda_N^2}{(2\pi)^2|k|^2}\sum_{p\in\ZZ_0^d}\frac{\fR^N_{p,k-p}(|\hat\Lp(p)k|^2|\hat\Lp(k-p)a_{k,1}|^2+(k\cdot\hat\Lp(p)a_{k,1})(k\cdot\hat\Lp(k-p)a_{k,1}))}{|p|^2+|k-p|^2}\nonumber\\
		=&\frac{\lambda^2}{(2\pi)^2|k|^2}\int_{|x|\leq1}\frac{|\hat\Lp(x)k|^2|\hat\Lp(x)a_{k,1}|^2+(k\cdot\hat\Lp(x)a_{k,1})^2}{2|x|^2}\ud x\,.\label{5.7}
	\end{align}
When $d=3$, for $x\in\RR^3$, let $x'$ denote the projection of $x$ onto span$\{a_{k,1},a_{k,2}\}$. Set $\theta_1$ be the angle between $x$ and $k$, and $\theta_2$ be the angle between $x'$ and $a_{k,1}$. Let $\alpha$ be the angle between $x$ and $a_{k,1}$. Then $\cos\alpha=\sin\theta_1\cos\theta_2$, and
\begin{align}
\lim_{N\to\infty}\|T^{N,+}\sigma_{k,1}\|^2=	&\frac{\lambda^2}{(2\pi)^2}\int_{|x|\leq1}\frac{\sin^2\theta_1\sin^2\alpha+\cos^2\theta_1\cos^2\alpha}{2|x|^2}\ud x\nonumber\\
	=&\frac{\lambda^2}{(2\pi)^2}\int_{|x|\leq1}\frac{\sin^2\theta_1(1+\cos2\theta_1\cos^2\theta_2)}{2|x|^2}\ud x\nonumber\\
	=&\frac{\lambda^2}{2(2\pi)^2}\int_0^{2\pi}\int_0^{\pi}\int_0^1\sin^3\theta_1(1+\cos2\theta_1\cos^2\theta_2)\ud r\ud \theta_1\ud\theta_2
	=\frac{7\lambda^2}{30\pi}\,,\label{5.6}
\end{align}
which is strictly less than $\nu_{\mathrm{eff}}-1=\sqrt{1+\frac{\lambda^2}{\pi}}-1$ in \cite[Conjecture 6.5]{JP24} for $\lambda\leq4$. 
}

	(ii){ By \eqref{**} for $m=2$, we have \begin{align*}
		D = \lambda^2\lim_{N\to \infty}\|T^{N,+,*}\sigma_{k,1}\|^2- \lambda^4\lim_{n\to\infty}\lim_{N\to\infty}\|T^{N,*}_{2,n}T^{N,+,*}\sigma_{k,1}\|^2+O(\lambda^6)\,.
	\end{align*}
Recall that $T^{N,*}_{2,n} = P_{2,n}T^{N,*}P_{2,n}$. Since 
$T^{N,+,*}\sigma_{k,1}\in\FK_2$, for any $n>2$, $T^{N,*}_{2,n}T^{N,+,*}\sigma_{k,1}=T^{N,+,*}T^{N,+,*}\sigma_{k,1}.$ Thus we have
\begin{align}
	D=\lambda^2\lim_{N\to \infty}\|T^{N,+,*}\sigma_{k,1}\|^2- \lambda^4\lim_{N\to\infty}\|T^{N,+,*}T^{N,+,*}\sigma_{k,1}\|^2+O(\lambda^6)\,.\label{series}
\end{align}
For the constant $D_{\mathrm{rep}}$ from Remark \ref{rmk},
since it is the unique positive solution to $x(1+x)=c\lambda^2$ with $c=\lim_{N\to \infty}\|T^{N,+,*}\sigma_{k,1}\|^2,$ we have
$$D_{\mathrm{rep}} = \frac{\sqrt{1+4\lim_{N\to \infty}\|T^{N,+,*}\sigma_{k,1}\|^2\lambda^2}-1}{2}.$$
Then by a Taylor expansion at $\lambda=0$,
$$D_{\mathrm{rep}} = \lambda^2\lim_{N\to \infty}\|T^{N,+,*}\sigma_{k,1}\|^2 - \lambda^4\lim_{N\to \infty}\|T^{N,+,*}\sigma_{k,1}\|^4 + O(\lambda^6)\,.$$
Assume that $D=D_{\mathrm{rep}}$, then one must have $\lim_{N\to \infty}\|T^{N,+}T^{N,+}\sigma_{k,1}\|^2 = \lim_{N\to \infty}\|T^{N,+}\sigma_{k,1}\|^4.$}
Recall the definition of $L^2_n$ for $n\in\NN$ in \eqref{L^2_n} and the inner-product in $\FK$ in \eqref{FK}.
For any $f\in\FK_2$ and $g\in\FK_3$, there holds that
\begin{align*}
	\langle \fA^N_+f,g\rangle_{L^2_3} = \frac{1}{3!} \langle\fA^N_+f,g\rangle=-\frac{1}{3!} \langle f,\fA^N_-g\rangle = -\frac{1}{3} \langle f,\fA^N_-g\rangle_{L^2_2}\,.
\end{align*}
Then we have
\begin{align}
	&\|T^{N,+}T^{N,+}\sigma_{k,1}\|^2\label{5.4}\\
	=&3!\sum_{k_{1:3},\alpha_{1:3}}\left|\langle T^{N,+}T^{N,+}\sigma_{k,1},(\sigma_{k_1,\alpha_1}\otimes \sigma_{k_2,\alpha_2}\otimes \sigma_{k_3,\alpha_3})_{\text{sym}}\rangle_{L^2_3}\right|^2\nonumber\\
	=& \frac{2}{3(2\pi)^4|k|^2} \sum_{k_{1:3},\alpha_{1:3}}\frac{1}{|k_{1:3}|^2}\left|\langle (-\fL_0)^{-1}\fA_+^N\sigma_{k,1},\fA_-^N(\sigma_{k_1,\alpha_1}\otimes \sigma_{k_2,\alpha_2}\otimes \sigma_{k_3,\alpha_3})_{\text{sym}}\rangle_{L^2_2}\right|^2\nonumber\,,
\end{align}
where 
\begin{align}\label{sym}
	(\sigma_{k_1,\alpha_1}\otimes \sigma_{k_2,\alpha_2}\otimes \sigma_{k_3,\alpha_3})_{\text{sym}} = \frac{1}{3!}\sum_{\sigma\in P_3}\sigma_{k_{\sigma(1)},\alpha_{\sigma(1)}}\otimes \sigma_{k_{\sigma(2)},\alpha_{\sigma(2)}}\otimes \sigma_{k_{\sigma(3)},\alpha_{\sigma(3)}},
\end{align}
and $P_3$ denotes the set of permutations of $(1,2,3)$.
By \eqref{A_-}, the above equation equals
\begin{align*}
	&\frac{6}{(2\pi)^4|k|^2} \sum_{k_{1:3},\alpha_{1:3}}\frac{1}{|k_{1:3}|^2}\left|\sum_{q=1}^2\langle (-\fL_0)^{-1}\fA_+^N\sigma_{k,1},\fA_-^N[q](\sigma_{k_1,\alpha_1}\otimes \sigma_{k_2,\alpha_2}\otimes \sigma_{k_3,\alpha_3})_{\text{sym}}\rangle_{L^2_2}\right|^2\,.
\end{align*}
Recall the definition of $\fA_-^N[q]$ in \eqref{Aq}. Then 
\begin{align}
	&\sum_{q=1}^2\langle (-\fL_0)^{-1}\fA_+^N\sigma_{k,1},\fA_-^N[q](\sigma_{k_1,\alpha_1}\otimes \sigma_{k_2,\alpha_2}\otimes \sigma_{k_3,\alpha_3})\rangle_{L^2_2}\nonumber\\
	=&\lambda_N^2\1_{\sum_{i=1}^3k_i=k}\fR^N_{k_1,k_2}\fR^N_{k_1+k_2,k_3}\frac{k_1\cdot a_{k_2,\alpha_2}}{|k_1+k_2|^2+|k_3|^2}\left((k\cdot \hat\Lp(k_1+k_2)a_{k_1,\alpha_1}) (a_{k_3,\alpha_3}\cdot a_{k,1})\right.\nonumber\\
	&\left.+(a_{k,1}\cdot \hat\Lp(k_1+k_2)a_{k_1,\alpha_1}) (a_{k_3,\alpha_3}\cdot k)\right)\nonumber\\
	=:&\lambda_N^2\1_{\sum_{i=1}^3k_i=k}\fR^N_{k_1,k_2}\fR^N_{k_1+k_2,k_3}C_{k_{1:3},\alpha_{1:3}}\,.\label{C}
\end{align}
Substituting the above equation into \eqref{5.4}. By \eqref{sym} and performing the change of variable $m_{1:3}=k_{\sigma'(1:3)}$ and $l_{1:3}=\alpha_{\sigma'(1:3)}$, we have
\begin{align*}
	&\|T^{N,+}T^{N,+}\sigma_{k,1}\|^2\\
	=&\frac{\lambda_N^4}{6(2\pi)^4|k|^2}\sum_{\sigma',\sigma\in P_3} \sum_{k_{1:3},\alpha_{1:3}}\frac{\1_{\sum_{i=1}^3k_i=k}}{|k_{1:3}|^2}\fR^N_{k_{\sigma'(1)},k_{\sigma'(2)}}\fR^N_{k_{\sigma'(1)}+k_{\sigma'(2)},k_{\sigma'(3)}}
	\fR^N_{k_{\sigma(1)},k_{\sigma(2)}}\fR^N_{k_{\sigma(1)}+k_{\sigma(2)},k_{\sigma(3)}}\times\\
	&C_{k_{\sigma'(1:3)},\alpha_{\sigma'(1:3)}}C_{k_{\sigma(1:3)},\alpha_{\sigma(1:3)}}\\
	=&\frac{\lambda_N^4}{(2\pi)^4|k|^2}\sum_{\sigma\in P_3} \sum_{m_{1:3},l_{1:3}}\frac{\1_{\sum_{i=1}^3m_i=k}}{|m_{1:3}|^2}\fR^N_{m_1,m_2}\fR^N_{m_1+m_2,m_3}
	\fR^N_{m_{\sigma(1)},m_{\sigma(2)}}\fR^N_{m_{\sigma(1)}+m_{\sigma(2)},m_{\sigma(3)}}\times\\
	&C_{m_{1:3},l_{1:3}}C_{m_{\sigma(1:3)},l_{\sigma(1:3)}}\,.
\end{align*}
Recall the definition of $\lambda_N$ in \eqref{lambda}. 
By the definition of $C_{m_{1:3},l_{1:3}}$ in \eqref{C}, we have $C_{m_{1:3}/N,l_{1:3}}=NC_{m_{1:3},l_{1:3}}$. Then
by taking $v_i=N^{-1}m_i$ for $i=1,2,3$, as $N\to\infty$, we have
\begin{align}
	&\lim_{N\to \infty}\|T^{N,+}T^{N,+}\sigma_{k,1}\|^2\nonumber\\
	=&\lim_{N\to \infty}\frac{\lambda^4}{(2\pi)^4|k|^2}\frac{1}{N^6}\sum_{\sigma\in P_3}\sum_{l_{1:3}=1}^3\sum_{\substack{v_1+v_2+v_3=k/N\\ v_{1:3}\in N^{-1}\ZZ_0^d}} \frac{\fR^1_{v_1,v_2}\fR^1_{v_1+v_2,v_3}	\fR^1_{v_{\sigma(1)},v_{\sigma(2)}}}{|v_{1:3}|^2}C_{v_{1:3},l_{1:3}}C_{v_{\sigma(1:3)},l_{\sigma(1:3)}}\nonumber\\
	=&\lim_{N\to \infty}\frac{\lambda^4}{(2\pi)^4|k|^2}\sum_{\sigma\in P_3} \sum_{l_{1:3}=1}^3\int_{|x_1|,|x_2|,|x_1+x_2|\leq1}\frac{C_{x_{1:3},l_{1:3}} C_{x_{\sigma(1:3)},l_{\sigma(1:3)}}}{|x_1|^2+|x_2|^2+|x_1+x_2|^2}\ud x_1\ud x_2\,,\label{5.8}
\end{align}
where $x_3=\frac{k}{N}-x_1-x_2$.
By direct evaluation of the integral via \texttt{Mathematica} shows that the limit is approximately $\frac{8.588}{2(2\pi)^4} \lambda^4$. This is strictly less than $\lambda^4 \lim_{N\to \infty}\|T^{N,+} \sigma_{k,1}\|^4$, whose value by \eqref{5.6} is $\left(\frac{7}{30\pi}\right)^2 \lambda^4$.  
Therefore, $D \neq D_{\mathrm{rep}}$.
\end{proof} 

\begin{remark}
	When $d=3$, the proof of Theorem \ref{thm:main} follows from an expansion for $(-\fL_{2,n}^N)^{-1}$.
	Such argument also holds in two dimensional case. In fact,
	when $d=2$, one can adapt the argument of \cite[Lemma~3.4]{CET23b} together with Proposition~\ref{prop:A} to show that for any $\psi,\varphi\in\FK_n$, and $M\in\NN$,
	\begin{align}\label{rep}
		\langle(-T^{N,-}((32\pi)^{-1}L^N(-\fL_0))^MT^{N,+}-\frac{1}{M+1}((32\pi)^{-1}L^N(-\fL_0))^{M+1})\psi,\varphi\rangle\lesssim_M\frac{1}{\log N}\|\varphi\|\|\psi\|,
	\end{align} 
where $L^N(x)=\frac{\lambda^2}{\log N}\log\left(1+N^2x^{-1}\right).$
	Thus we can use the diagonal operator $L^N(-\fL_0)$ to compute the limit of  $\langle(T^N_{2,n})^mT^{N,+}\sigma_{k,1},$
	$T^{N,+}\sigma_{-k,1}\rangle$ in the series in \eqref{C.3}.
{In this case, the Taylor expansion of $D$ w.r.t $\lambda$ in \eqref{series} conincides with the expansion of $D_{\mathrm{rep}}$. More pricisely, by the Replacement Lemma \ref{Lem5.1} and \eqref{rep}, we can obtain that $$D_{\mathrm{rep}}=\lim_{N\to \infty}G(L^N((2\pi)^2|k|^2))= \sqrt{\frac{\lambda^2}{8\pi}+1}-1 = \sqrt{2\lim_{N\to \infty}\|T^{N,+,*}\sigma_{k,1}\|^2\lambda^2+1}-1.$$
	Then by Taylor expansion w.r.t $\lambda$ at $\lambda=0$,
	we have
	\begin{align}\label{series1}
		D_{\mathrm{rep}}=\lambda^2\lim_{N\to \infty}\|T^{N,+,*}\sigma_{k,1}\|^2 -\lambda^4\lim_{N\to \infty}\frac12\|T^{N,+,*}\sigma_{k,1}\|^4+O(\lambda^6).
	\end{align}
While by \eqref{rep}, 
$$\lim_{N\to \infty}\|T^{N,+,*}T^{N,+,*}\sigma_{k,1}\|^2=\lim_{N\to \infty}\frac{1}{2}((32\pi)^{-1}L^N((2\pi)^2|k|^2))^2 = \frac12\lim_{N\to \infty}\|T^{N,+,*}\sigma_{k,1}\|^4.$$
More caculation also implies the Taylor expansion of $D$ w.r.t $\lambda$ in \eqref{series} conincides with the expansion of $D_{\mathrm{rep}}$ in \eqref{series1}.}
	
	For $d\geq3$, this is not the case, since when we consider the action of $T^{N,-}T^{N,+}$, there is no weak coupling constant $\frac{1}{\log N}$ on R.H.S. of \eqref{4.1}
	such that the off-diagonal part can be omitted.
\end{remark}

\section{$d=2$: the Replacement Lemma}\label{sec6}
When $d=2$, similar as \cite{CET23b} and \cite{CGT24},
we can approximate $(-\fL^N)^{-1}$ as $N\to\infty$ by a diagonal operator, as established by the following Replacement Lemma.
Let $L^N$ be the function defined on $[(2\pi)^2,\infty)$ as
$L^N(x):=\lambda_N^2\log\left(1+N^2x^{-1}\right),$
and $\fG^N$ be the operator on $\FK$ given by
$\fG^N:=G(L^N(-\fL_0)),$ where 
\begin{align}\label{eq:G}
	G(x)=\sqrt{\frac{x}{16\pi}+1}-1\,.
\end{align}
Now we prove the replacement lemma.
{Compared with scalared euqation such as the anisotropic KPZ equation or the stochastic Burgers equation, the action of $\fA^N_+$ on $\psi$ in \eqref{3.10} has the Leray projection and generates different terms in the vector components $l_i$ and $l_j$. As a consequence, the diagonal part of $\langle \fS^N \fA_{+}^N\psi_1,\fA_{+}^N\psi_2\rangle$, defined in \eqref{3.27}, {consists of two terms (see \eqref{e:Diag} and \eqref{e:Diag_2} below), one involving the inner product between $\hat\psi_1$ and $\hat\psi_2$, and the other involving the inner product between $\hat\psi_i$ and the momentum.} For the second term, we use the divergence-free properties of $\psi_1$ and $\psi_2$ to convert the latter inner-product into one between $\hat\psi_1$ and $\hat\psi_2$ in \eqref{6.13} below.}
The remaining steps of the proof then follow by an argument similar to that in \cite{CGT24}.

\begin{lem}[Replacement Lemma]\label{Lem5.1}
There exists a constant $C>0$ independent of $N$ such that for every $\psi_1,\psi_2\in\FK_n\cap \fS(\TT^{2n},\RR^{2n})$ and $n\in\NN$, we have 
\begin{align}
|\langle[-\fA_{-}^N(-\fL_0-\fL_0\fG^N)^{-1}\fA_{+}^N+\fL_0\fG^N]\psi_1,\psi_2\rangle|
\leq C\lambda_N^2\|\fN(-\fL_0)^{\frac12}\psi_1\|\|\fN(-\fL_0)^{\frac12}\psi_2\|.\label{eq:rep}
\end{align}
\end{lem}
\begin{proof}
By $(\fA^N_+)^{*}=-\fA_-^N$, the first part of the left hand side of \eqref{eq:rep} can be written as 
\begin{align}\label{3.26}
\langle(-\fL_0-\fL_0\fG^N)^{-1}\fA_{+}^N\psi_1,\fA_{+}^N\psi_2\rangle.
\end{align}
Let $\fS^N:=(-\fL_0-\fL_0\fG^N)^{-1}$ and 
denote $\sigma^N =(\sigma^N_n)_{n\ge1}$ as its Fourier multiplier, i.e. for $\psi\in\Gamma L^2_n$ 
$\fF(\fS^N \psi)(l_{1:n},k_{1:n})=\sigma^N _n(k_{1:n})\hat \psi(l_{1:n},k_{1:n})$, where 
\begin{align}\label{sigma}
	\sigma^N _n(k_{1:n})&=\frac1{(2\pi)^2|k_{1:n}|^2(1+G(L^N((2\pi)^2|k_{1:n}|^2)))}\,.
\end{align}
Bt the form of $\fA_{+}^N$ in \eqref{3.10}, \eqref{3.26} can be splitted into two {diagonal parts}, corresponding to making the same or inverse choice for the indices $i,j$ in the two occurrences of $\fA_{+}^N$, 
an {off-diagonal part of type 1}, corresponding to $\{i=i',j\neq j'\}$ or $\{j=j',i\neq i'\}$ or $\{i=j',j\neq i'\}$ or $\{j=i',i\neq j'\}$, 
and {off-diagonal part of type 2}, corresponding to $i,j,i',j'$ are all different. 
Recall that for $k\in\ZZ_0^d,\, x,y\in\RR^d$, $x\cdot(\hat\Lp(k)y)=x\cdot y-\frac{(x\cdot k)(y\cdot k)}{|k|^2}$\,. Then \eqref{3.26} can be written as 
\begin{align}\label{3.27}
	\langle \fS^N \fA_{+}^N\psi_1,\fA_{+}^N\psi_2\rangle=\sum_{i=1}^2\langle \fS^N \fA_{+}^N\psi_1,\fA_{+}^N\psi_2\rangle_{\text{diag}_i} +\sum_{i=1}^2 \langle \fS^N \fA_{+}^N\psi_1,\fA_{+}^N\psi_2\rangle_{\text{off}_i},
\end{align}
where the diagonal part 1, corresponding to $\{i=i',j=j^{'}\}$, is defined as 
\begin{align}\label{e:Diag}
	&\langle \fS^N \fA_{+}^N\psi_1,\fA_{+}^N\psi_2\rangle_{\text{diag}_1}
	:={n!n(2\pi)^2\lambda_N^2}\sum_{l_{2:n}} \sum_{k_{1:n}}\sum_{l+m=k_1}\sigma^N_{n+1}(l,m,k_{2:n})\fR_{l,m}^N\left[|k_1|^2-\frac{(k_1\cdot m)^2}{|m|^2}\right]\\
	&\left[\overline{\hat\psi_1((\cdot,k_1),(l_{2:n},k_{2:n}))}\cdot{\hat\psi_2((\cdot,k_1),(l_{2:n},k_{2:n}))}-\frac{(l\cdot\overline{\hat\psi_1((\cdot,k_1),(l_{2:n},k_{2:n}))})(l\cdot{\hat\psi_2((\cdot,k_1),(l_{2:n},k_{2:n}))})}{|l|^2}\right]\nonumber,
\end{align}
and the diagonal part 2, corresponding to $\{i=j^{'},j=i'\}$, is defined as 
\begin{align}
	\langle \fS^N \fA_{+}^N\psi_1,\fA_{+}^N\psi_2\rangle_{\text{diag}_2}:=& {n!n(2\pi)^2\lambda_N^2}\sum_{l_{2:n}}\sum_{k_{1:n}}\sum_{l+m=k_1}\sigma^N_{n+1}(l,m,k_{2:n}) (\fR_{l,m}^N)^2 \times\label{e:Diag_2}\\
	&\left[k_1\cdot\hat\psi_2((\cdot,k_1),(l_{2:n},k_{2:n}))-\frac{(k_1\cdot m)(\hat\psi_2((\cdot,k_1),(l_{2:n},k_{2:n}))\cdot m)}{|m|^2}\right]\times\nonumber\\
	&\left[k_1\cdot\hat\psi_1((\cdot,k_1),(l_{2:n},k_{2:n}))-\frac{(k_1\cdot l)(\hat\psi_1((\cdot,k_1),(l_{2:n},k_{2:n}))\cdot l)}{|l|^2}\right].\nonumber
\end{align}
Since $\psi_i$ is devergence-free,  $k_1\cdot{\hat\psi_i((\cdot,k_{1}),(l_{2:n},k_{2:n}))}=\fF(\nabla\cdot\psi_i((\cdot,x_{1}),(l_{2:n},k_{2:n})))(k_1)=0.$
Then we have
\begin{align}
	&\sum_{i=1}^2\langle \fS^N \fA_{+}^N\psi_1,\fA_{+}^N\psi_2\rangle_{\text{diag}_i}={n!n(2\pi)^2\lambda_N^2}\sum_{l_{2:n}}\sum_{k_{1:n}}\sum_{l+m=k_1}\sigma^N_{n+1}(l,m,k_{2:n}) (\fR_{l,m}^N)^2 \times\nonumber\\
	&\left(\left[|k_1|^2-\frac{(k_1\cdot m)^2}{|m|^2}\right]\overline{\hat\psi_1((\cdot,k_1),(l_{2:n},k_{2:n}))}\cdot{\hat\psi_2((\cdot,k_1),(l_{2:n},k_{2:n}))}\right.\nonumber\\
	&\left.-\frac{|m|^2|k_1|^2-(k_1\cdot m)^2+(k_1\cdot m)(k_1\cdot l)}{|m|^2|l|^2}(l\cdot\overline{\hat\psi_1((\cdot,k_1),(l_{2:n},k_{2:n}))})(l\cdot{\hat\psi_2((\cdot,k_1),(l_{2:n},k_{2:n}))})\right).\label{4.8}
\end{align}
Since $|x\cdot(\hat\Lp(k)y)|\lesssim|x||y|$, off-diagonal parts of type $1$ and $2$ respectively have upper bounds
\begin{align}
	|\langle \fS^N \fA_{+}^N\psi_1,\fA_{+}^N\psi_2\rangle_{\text{off}_1}|
		\lesssim&\lambda_N^2n!n(n-1)\sum_{k_{1:n+1}} \sigma^N_{n+1}(k_{1:n+1})\fR^N_{k_1,k_2}\fR^N_{k_1,k_3}\times\nonumber\\
	&|k_1+k_2||k_1+k_3||\hat{\psi}_1(k_1+k_2,k_{3:n+1})||{\hat\psi_2(k_1+k_3,k_2,k_{4:n+1})}|,\label{3.28}
\end{align}
and
\begin{align}
	|\langle \fS^N \fA_{+}^N\psi_1,\fA_{+}^N\psi_2\rangle_{\text{off}_2}|\lesssim&\lambda_N^2n!n(n-1)(n-2)\sum_{k_{1:n+1}} \sigma^N_{n+1}(k_{1:n+1})\fR^N_{k_1,k_2}\fR^N_{k_3,k_4}\times\nonumber\\
	&|k_1+k_2||k_3+k_4||\hat{\psi}_1(k_1+k_2,k_{3:n+1})||\hat\psi_2(k_3+k_4,k_{1:2},k_{5:n+1})|.\label{3.29}
\end{align}
Note that the right hand side of \eqref{3.28} and \eqref{3.29} have a similar structure as \cite[(2.27) and (2.28)]{CGT24}. Then by the identical argument, we can derive that
\begin{align}
	\sum_{i=1}^2|\langle\fS^N \fA_{+}^N\psi_1, \fA_{+}^N\psi_2\rangle_{\mathrm{off}_i}|\lesssim\lambda_N^2n^2\|(-\fL_0)^{\frac12}\psi_1\|\|(-\fL_0)^{\frac12}\psi_2\|.
\end{align}

It suffices to estimate
\begin{align}
\Big|\sum_{i=1}^2\langle\fS^N \fA_{+}^N\psi_1, \fA_{+}^N\psi_2\rangle\rangle_{\text{diag}_i}+ \langle \fL_0\fG^N\psi_1,\psi_2\rangle\Big|\,.\label{e:Di}
\end{align}
We focus on \eqref{4.8} first.
Let $\theta_1,\theta_2$ be the angles between $k_1$ and $m$, $l$ respectively. By \eqref{eq:div}, $\hat\psi_i((\cdot,k),(l_{2:n},k_{2:n}))=(\psi_i(\cdot,(l_{2:n},k_{2:n})))_{k,1}a_{k,1}$ for fixed $l_{2:n}$ and $k_{2:n}$.
Recall that $a_{k,1}=a_{-k,1}$.
Then we have
\begin{align}
	&(l\cdot\overline{\hat\psi_1((\cdot,k_1),(l_{2:n},k_{2:n}))})(l\cdot{\hat\psi_2((\cdot,k_1),(l_{2:n},k_{2:n}))})\nonumber\\
	=&(\psi_1(\cdot,(l_{2:n},-k_{2:n})))_{-k_1,1}(\psi_2(\cdot,(l_{2:n},k_{2:n})))_{k_1,1}(l\cdot a_{k_1,1})^2\nonumber\\
	=&|l|^2\sin^2\theta_2\overline{\hat\psi_1((\cdot,k_1),(l_{2:n},k_{2:n}))}\cdot\hat\psi_2((\cdot,k_1),(l_{2:n},k_{2:n}))\,.\label{6.13}
\end{align}
where we used that $a_{k_1,1}$ is a unit vector vertical to $k_1$.
Thus \eqref{4.8} equals
\begin{align*}
	{n!n(2\pi)^2\lambda_N^2}\sum_{l_{1:n}}\sum_{k_{1:n}}&|k_1|^2\overline{\hat\psi_1(l_{1:n},k_{1:n})}{\hat\psi_2(l_{1:n},k_{1:n})}\times\\
	&\sum_{l+m=k_1}\sigma^N_{n+1}(l,m,k_{2:n})\fR_{l,m}^N \left[\sin^2\theta_1-\sin^2\theta_2\left(\sin^2\theta_1+\frac{|l|}{|m|}\cos\theta_1\cos\theta_2\right)\right].
\end{align*}
Substitute it into \eqref{e:Di} and we obtain
\begin{align}
	&\left|n!n(2\pi)^2\sum_{l_{1:n}}
	\sum_{k_{1:n}}|k_1|^2\overline{\hat\psi_1(l_{1:n},k_{1:n})}{\hat\psi_2(l_{1:n},k_{1:n})}\left[\lambda_N^2\sum_{l+m=k_1}\sigma^N_{n+1}(l,m,k_{2:n}) \fR_{l,m}^N\times\right.\right.\nonumber\\
	&\left.\left.\left(\sin^2\theta_1-\sin^2\theta_2\left(\sin^2\theta_1+\frac{|l|}{|m|}\cos\theta_1\cos\theta_2\right)\right)-G(L^N((2\pi)^2|k_{1:n}|^2))\right]\right|.\label{4.14}
\end{align}
By Proposition \ref{prop:A} and the Cauchy-Schwarz inequality, there exists a constant $C$ such that \eqref{4.14} is bounded by
\begin{align}
	C\lambda_N^2n!n(2\pi)^2\sum_{l_{1:n}}
	\sum_{k_{1:n}}|k_1|^2|{\hat\psi_1(l_{1:n},k_{1:n})}||{\hat\psi_2(l_{1:n},k_{1:n})}|\lesssim
	\lambda_N^2\|(-\fL_0)^{\frac12}\psi_1\|\|(-\fL_0)^{\frac12}\psi_2\|.\nonumber
\end{align}
Thus the proof is complete.
\end{proof}
By the Replacement Lemma, we can employ the approximating operator $-\fL_0\fG^N$ to construct $v^{N,n}$ in Theorem \ref{thm:F-D}, which is an approximation to the solution $u^{N,n}$ to \eqref{eq:geneq}.
For fixed $i=2$ or $3$, $N>0,$ and $i\leq n\in\NN$, define $v^{N,n}\in\oplus_{j=i}^n\FK_j$ as
\begin{align*}
	v^{N,n}=(-\fL_0-\fL_0\fG^N)^{-1}\fA_{+}^NP^{n-1}_iv^{N,n}
	+(-\fL_0-\fL_0\fG^N)^{-1}\fA_{+}^Nf_{i-1},
\end{align*}
where $f_{i-1}$ is given by \eqref{eq:fi}.
{Then following the argument in \cite[Theorem 2.7 and Theorem 2.11]{CGT24}
we can prove Theorem \ref{thm:F-D} when $d=2$ and derive the constant $D$ in \eqref{eq:limeq}.}

\appendix
\renewcommand{\appendixname}{Appendix~\Alph{section}}
\renewcommand{\theequation}{A.\arabic{equation}}
\section{The generator on the Fock space}\label{appenA}
In this appendix, we prove Theorem \ref{thm.Geneator_in_Fourier} and Lemma \ref{lem:sector} using a similar argument as \cite{CES21} and \cite{CGT24}.

\begin{proof}[Proof of Theorem \ref{thm.Geneator_in_Fourier}]
	By linearity, we only need to consider $F=I_n(h^{\otimes n})= H_n(\eta(h))$ with $h \in \HH\cap\fS(\TT^d,\RR^d)$ such that $\|h\|=1$. 
	By It\^o's formula, we obtain
	\begin{align}
		\fL_0 F &= \sum_{k\in\ZZ_0^d}(2\pi|k|)^2(-\hat{\eta}(-k)\cdot D_k+D_{-k}\cdot D_k)F,\label{Malli1}\\
		\fA^N F&=-\lambda_N\iota2\pi\sum_{l,m\in\ZZ_0^d}\mathcal{R}_{l,m}^N\left[(l+m)\cdot(\hat{\eta}(l)\otimes\hat{\eta}(m))\right]\cdot D_{-l-m}F.\label{Malli2}
	\end{align}
{Following the arguments in the proof of \cite[Lemma 3.5]{CES21}, we can show that $\fL_0I_n(h^{\otimes n})=I_n(\Delta h^{\otimes n})$.}
Set $\tilde{e}^l_k$ be the $d-$dim vector whose $l$-th component is $e_k$ and others are 0.
	For $\fA^N$, by \eqref{Malli2} we have 
	\begin{align*}
		\fA^NI_n(h^{\otimes n})=
		-\lambda_N\iota2\pi nI_{n-1}(h^{\otimes (n-1)})\sum_{i,j=1}^d\sum_{l,m\in\ZZ_0^d}\mathcal{R}_{l,m}^N(l+m)_iI_1(\Lp\tilde{e}_{-l}^i)I_1(\Lp\tilde{e}_{-m}^j)\hat{h}(j,-l-m).
	\end{align*}
	Recall from \cite{Nua06} that for any $\varphi_p\in\FK_p$ and $\varphi_q\in\FK_q$,
\begin{equation}\label{chainrule}
	I_p(\varphi_p)I_q(\varphi_q) = \sum_{r=0}^{p\wedge q}r!\binom{p}{r}\binom{q}{r} I_{p+q-2r}(\text{sym}(\varphi_p\otimes_{r}\varphi_q)),
\end{equation}
where $\text{sym}(\cdot)$ is the symmetrization of the function in $\FK$ and 
\begin{align*}
	&\varphi_p\otimes_r \varphi_q((i_{1},x_{1}),\dots,(i_{p+q-2r},x_{p+q-2r})) \\
	&\hspace{5pt}= \langle \varphi_p((i_1,x_1),\dots,(i_{p-r},x_{p-r}),\cdot), \varphi_q((i_{p-r+1},x_{p-r+1}),\dots,(i_{p+q-2r},x_{p+q-2r}),\cdot)\rangle_{L^2(\TT^d,\RR^d)^{\otimes r}}. 
\end{align*}
	By \eqref{chainrule} and $\langle\Lp\tilde{e}^i_{-l}, \Lp\tilde{e}^j_{-m}\rangle_{L^2(\TT^d,\RR^d)}$ vanishes if $l\neq-m$, the above equation equals
	\begin{align*}
		-\lambda_N\iota2\pi nI_{n-1}(h^{\otimes (n-1)})I_2\left(\sum_{i,j}\sum_{l,m\in\ZZ_0^d}\fR^N_{l,m}(l+m)_i\text{sym}(\Lp\tilde{e}^i_{-l}\otimes \Lp\tilde{e}^j_{-m})\hat{h}(j,-l-m)\right).
	\end{align*}
	Again by the product rule \eqref{chainrule}, we get
	\begin{align*}
		\fA^NI_n(h^{\otimes n})
		=&-\lambda_N\iota2\pi n
		I_{n+1}\left(\sum_{i,j}\sum_{l,m\in\ZZ_0^d}\fR^N_{l,m}(l+m)_i\hat{h}_j(-l-m)\text{sym}(h^{\otimes (n-1)}\otimes \Lp\tilde{e}^i_{-l}\otimes \Lp\tilde{e}^j_{-m})\right)\\
		&-\lambda_N\iota2\pi n(n-1)
		I_{n-1}\left(\sum_{i,j}\sum_{l,m\in\ZZ_0^d}\fR^N_{l,m}(l+m)_i\hat{h}_j(-l-m)\times\right.\\
		&\hspace{100pt}\text{sym}\left.(h^{\otimes (n-2)}\otimes \Lp\tilde{e}^i_{-l} \hat{h}_j(m)+h^{\otimes (n-2)}\otimes \Lp\tilde{e}^j_{-m} \hat{h}_i(l))\right)\\
		&-\lambda_N\iota2\pi n^2(n-1)
		I_{n-3}\left(\sum_{i,j}\sum_{l,m\in\ZZ_0^d}\fR^N_{l,m}(l+m)_i\hat{h}_j(-l-m)h^{\otimes (n-3)}\hat{h}_i(l)\hat{h}_j(m)\right).
	\end{align*}
	The third term disappears since $B^N(h)(h)=0.$
	For the first term, by taking the Fourier transform, \eqref{3.10} holds.
	For the second term, since $\fF(\Lp\tilde{e}^i_{-l})(l_{n-1},k_{n-1})=\1_{k_{n-1}=-l}\hat\Lp_{l_{n-1},i}(k_{n-1})$, we have
	\begin{align*}
		&\fF(\fA^N_{-}h^{\otimes n})(l_{1:n-1},k_{1:n-1})\\
		=&\lambda_N\iota2\pi n(n-1)
		\text{sym}\left(\sum_{i,j}\sum_{p+q=k_{n-1}}\fR^N_{p,q}\hat\Lp_{l_{n-1},i}(k_{n-1})p_{i}\hat{h}_{j}(p)\hat{h}_{j}(q)\widehat{h^{\otimes(n-2)}}(l_{1:n-2},k_{1:n-2})+\right.\\
		&\left.\hspace{90pt}+\sum_{i,j}\sum_{p+q=k_{n-1}}\fR^N_{p,q}\hat\Lp_{l_{n-1},i}(k_{n-1})p_j\hat{h}_{i}(p)\hat{h}_{j}(q)\widehat{h^{\otimes(n-2)}}(l_{1:n-2},k_{1:n-2})\right).
	\end{align*}
	By interchanging $p$ and $q$ in the summation and using the fact that $\hat\Lp(k_{n-1})k_{n-1}=0$, the first term in the above summation vanishes.
	Since $h\in\HH,$ $\sum_j p_j\hat h_j(q)=\sum_j k_{n-1}^j\hat h_j(q),$ and \eqref{3.11} holds.
	A direct calculation shows that  $\fA_-^{N} = 0$ for $n=0,1$, and $\fA_+^{N} = 0$ for $n=0$. 
	
	Due to the anti-symmetry of $\fA^N$, it follows from the last part of the proof in \cite[Lemma 3.5]{CES21} that $\fA^N_-=-(\fA^N_+)^{*}$. 
	\eqref{eq:commute} is immediate by verifying the Fourier transforms of both sides.     
\end{proof}
Next, we prove Lemma \ref{lem:sector}.
 \begin{proof}[Proof of Lemma \ref{lem:sector}]
	For \eqref{eq:sector}, arguing as \cite[Lemma 2.4]{CGT24}, by the variational formula,
	it suffices to prove that for all $\gamma>0$, $\varphi\in\FK_n,$ $\rho\in \FK_{n+1}$, we have
	\begin{align}\label{3.17_1}
		|\langle\rho,\fA^N_+\varphi\rangle|\lesssim\lambda\sqrt{n}\left(\gamma\|(-\fL_0)^\frac12\varphi\|^2+\frac1\gamma\|(-\fL_0)^\frac12\rho\|^2\right).
	\end{align} 
	Recall the fact that for $f\in\FK_{1},\, g\in L^2(\TT^d,\RR^d)$, $\sum_{l,k}\hat f(l,-k)(\hat\Lp(k)\hat g(k))(l)=\sum_{l,k}\hat f(l,-k)\hat g(l,k).$
	Since for the operator $\fA^N_+$, by \eqref{3.10}, the Leray projections act on $k_i+k_j$ and the first component of the Fourier transform, 
	the left hand side of \eqref{3.17_1} can be written as
	\begin{align*}
		\lambda_N 2\pi (n+1)!n\left|\sum_{l_{1:n+1}=1}^d\sum_{k_{1:n+1}\in\ZZ_0^d}\hat{\rho}(l_{1:n+1},-k_{1:n+1})\fR_{k_1,k_2}^N(k_1+k_2)_{l_2}\hat{\varphi}((l_1,k_1+k_2),(l_{3:n+1},k_{3:n+1}))\right|.
	\end{align*}
	Then the proof for \eqref{3.17_1} is identical to that of (2.12) in \cite[Lemma 2.4]{CGT24}.
	
	Now we prove \eqref{eq:A-}. Regarding the operator $\fA^N_-$, by \eqref{3.11}, the Leray projection acts on the first component of the Fourier transform, while $k_j$ interacts with the second component via the inner product. Since $|\hat\Lp(k)y|\leq|y|$ for any $k$ and $y$, $\|(- \fL_0)^{-\frac{1}{2}}\fA_{-}^{N}\varphi\|^2$ has the upper bound
	\begin{align*}
		\lambda_N^2(4\pi)^2\sum_{k\in\ZZ_0^d}\left|\sum_{p+q=k}\fR^N_{p,q}|\hat\varphi(p,q)|\right|^2=\lambda_N^2(4\pi)^2\sum_{k\in\ZZ_0^d}\left|\sum_{p+q=k}\frac{\fR^N_{p,q}}{(|p|^2+|q|^2)^{\frac\alpha2}}(|p|^2+|q|^2)^{\frac\alpha2}|\hat\varphi(p,q)|\right|^2,
	\end{align*}
	for $\alpha>1.$
	By the Cauchy-Schwarz inequality and the regularity of $\varphi$, the above equation is bounded by
	\begin{align*}
		(4\pi)^2\left(\lambda_N^2\sum_{1\leq|p|\leq N}\frac{1}{|p|^{2\alpha}}\right)\|(-\fL_0)^{\frac\alpha2}\varphi\|^2,
	\end{align*}
	which converges to 0 as $N\to\infty.$ Thus \eqref{eq:A-} follows.
\end{proof}

\renewcommand{\theequation}{B.\arabic{equation}}

\section{Technical Estimates in the Replacement Lemma}\label{appenA1}
In the following, we give the technical estimate of \eqref{4.14} in the Replacement Lemma.

\begin{prop}\label{prop:A}
	For $n,N\in\NN$, $\lambda>0$, and $k_{1:n}\in\ZZ_0^d$, let $P^N$ be 
	\begin{align}
		P^N(k_{1:n}):={\lambda_N^2}\sum_{l+m=k_1}\frac{\fR_{l,m}^N\left[\sin^2\theta_1-\sin^2\theta_2\left(\sin^2\theta_1+\frac{|l|}{|m|}\cos\theta_1\cos\theta_2\right)\right]}{(2\pi)^2(|l|^2+|m|^2+|k_{2:n}|^2)(1+G(L^N((2\pi)^2(|l|^2+|m|^2+|k_{2:n}|^2))))},\label{A1}
	\end{align}
where $G$ is defined in \eqref{eq:G}. $\theta_1$ and $\theta_2$ are the angles between $k_1$ and $m,\, l$ respectively. Then there exists a constant $C$ independent of $N$ such that
\begin{align}
	\sup_{k_{1:n}\in\ZZ_0^d}|P^N(k_{1:n})-G(L^N((2\pi)^2|k_{1:n}|^2))|\leq C\lambda_N^2\,.
\end{align}
\end{prop}
\begin{proof}
	Arguing as (A.3) of \cite{CGT24}, if $|k_1|>\frac N2$, we assume without loss of generality that $|l|\geq\frac N4$, and we have
	\begin{align*}
		|P^N(k_{1:n})|\lesssim&\lambda_N^2\sum_{\frac N4\leq|l|\leq N}\left(\frac{1}{|l|^2}+\frac{1}{|l||k_1-l|}\right)
		\lesssim\lambda_N^2\int_{\frac14\leq|x|\leq 1}\left(\frac{1}{|x|^2}+\frac{1}{|x||\frac{k_1}{N}-x|}\right)\lesssim \lambda_N^2.
	\end{align*}
	 Thus we only consider the case when $|k_1|\leq \frac N2$. For simplicity, for $x\in\RR^2$ and $|x|\leq1$, define $$\beta_N:=\left|\frac{k_{2:n}}{N}\right|^2,\,\quad\Gamma(x):=(2\pi)^2\left(|x|^2+\left|\frac{k_1}{N}-x\right|^2+\beta_N\right),\quad\Gamma_1(x):=(2\pi)^2\left(2|x|^2+\left|\frac{k_{1:n}}{N}\right|^2\right).$$
	 In the following steps, we simplify $P^N$ by resplcing it with several integrals, i.e. to construct $P^N_i$ and show that 
	 \begin{align}\label{A3}
	 	\sup_{k_{1:n}\in\ZZ_0^d}|P^N_i(k_{1:n})-P^N_{i+1}(k_{1:n})|\lesssim\lambda_N^2
	 \end{align}
	  for $i=0,\cdots,4$ and $P^N_0:=P^N.$
	  
	 \noindent {\bf Step 1} First, we convert the Riemann sum into an integral. Define $P^N_1(k_{1:n})$ as 
	 \begin{align}\label{PN1}
	 	{\lambda_N^2}\int \ud x\frac{\fR_{Nx,k_1-Nx}^N\left[\sin^2\theta_1-\sin^2\theta_2\left(\sin^2\theta_1+\frac{|x|}{|\frac{k_1}{N}-x|}\cos\theta_1\cos\theta_2\right)\right]}{\Gamma(x)(1+G(L^N(N^2\Gamma(x))))},
	 \end{align}
	 where $\theta_1,\,\theta_2$ are the angles between $N^{-1}k_1$ and $N^{-1}k_1-x,\, x$ respectively.
	 As argued in Step 4 in the proof of \cite[Proposition A.1]{CGT24} and \cite[Lemma C.7]{CET23a}, denote $I$ as the
	 integrand in \eqref{PN1} and $Q^N_l$ as the square of side-length $\frac1N$ centred at $N^{-1}l$ with $1\leq|l|\leq N$. Then we have
	 \begin{align*}
	 	\sup_{k_{1:n}\in\ZZ_0^d}|P^N(k_{1:n})-P^N_{1}(k_{1:n})|
	 	&\lesssim\lambda_N^2\sum_{1\leq|l|\leq N}\int_{Q^N_l}|I(N^{-1}l)-I(x)|\ud x
	 	\leq \lambda_N^2\frac{1}{N^3}\sum_{1\leq|l|\leq N}\sup_{y\in Q_l^N}|\nabla I(y)|.
	 \end{align*}
	 Since $$\sup_{y\in Q_l^N}|\nabla I(y)|\lesssim \frac{1}{|N^{-1}l|^3}+\frac{1}{|N^{-1}l|^2|N^{-1}k_1-N^{-1}l|}+\frac{1}{|N^{-1}l||N^{-1}k_1-N^{-1}l|^2},$$
	 \eqref{A3} holds for $i=0$.
	 
	 \noindent {\bf Step 2} As $N\to\infty$, for fixed $x\in\RR^2$, the difference between $\sin^2\theta_1-\sin^2\theta_2\left(\sin^2\theta_1+\frac{|x|\cos\theta_1\cos\theta_2}{|\frac{k_1}{N}-x|}\right)$ in \eqref{PN1} and $2\sin^2\theta_1\cos^2\theta_1$ vanishes. We define $P^N_2(k_{1:n})$ as 
	  \begin{align}\label{PN2}
	 	\lambda_N^2\int \ud x\frac{2\fR_{Nx,k_1-Nx}^N\sin^2\theta_1\cos^2\theta_1}{\Gamma(x)(1+G(L^N(N^2\Gamma(x))))}.
	 \end{align}
 Then we have
	 \begin{align}
	 	\sup_{k_{1:n}\in\ZZ_0^d}|P_1^N(k_{1:n})-P^N_{2}(k_{1:n})|
	 	\lesssim&\lambda_N^2\int \ud x\frac{\fR^N_{Nx,k_1-Nx}|\cos^2\theta_2-\cos^2\theta_1|}{\Gamma(x)}\nonumber\\
	 	&+\lambda_N^2\int \ud x\frac{\fR^N_{Nx,k_1-Nx}|\sin^2\theta_2\cos\theta_2|x|+\sin^2\theta_1\cos\theta_1|N^{-1}k_1-x||}{\Gamma(x)|N^{-1}k_1-x|}\nonumber\\
	 	=:&I+II.\label{A6}
	 \end{align}
	 
	 For the first integral $I$ in the right hand side of \eqref{A6}, as the proof of (A.10) in \cite{CET23b}, 
	 we split it over two regions corresponding to $|N^{-1}k_1-x|\leq\frac12|N^{-1}k_1|$ and $|N^{-1}k_1-x|>\frac12|N^{-1}k_1|.$
	 For the former, we have $|x|\in[\frac12|N^{-1}k_1|,\frac32|N^{-1}k_1|].$ Then the integral is bounded by a constant times $$\lambda_N^2\int_{\frac12|k_1/N|}^{\frac32|k_1/N|}\frac{1}{|x|^2}\ud x\lesssim \lambda_N^2.$$
	 For $I$ in the second region, since
	 \begin{align}
	 	|\cos^2\theta_1-\cos^2\theta_2|=&\left|\frac{(k_1\cdot(k_1/N-x))^2}{|k_1|^2|k_1/N-x|^2}-\frac{(k_1\cdot x)^2}{|k_1|^2|x|^2}\right|\nonumber\\
	 	\leq&\frac{|(k_1\cdot(k_1/N-x))^2-(k_1\cdot x)^2|}{|k_1|^2|k_1/N-x|^2}+\frac{(k_1\cdot x)^2}{|k_1|^2}\left|\frac{1}{|k_1/N-x|^2}-\frac{1}{|x|^2}\right|\nonumber\\
	 	\lesssim&\frac{|k_1/N||k_1/N-2x|}{|k_1/N-x|^2}\lesssim\frac{|k_1/N|}{|k_1/N-x|},\label{A8}
	 \end{align}
	 it has the bound
	 \begin{align}
	 	\lambda_N^2\int_{|k_1/N-x|>\frac12|k_1/N|}\frac{|k_1/N|}{|k_1/N-x|^3}\ud x\lesssim\lambda_N^2.\label{A9}
	 \end{align}
 
	 For the second integral $II$, we add and subtract a term from the integrand so that $II$ is bounded by
	 \begin{align}
	 	\lambda_N^2\int \ud x\frac{\fR^N_{Nx,k_1-Nx}||x|-|N^{-1}k_1-x||}{\Gamma(x)|N^{-1}k_1-x|}+\lambda_N^2\int \ud x\frac{\fR^N_{Nx,k_1-Nx}|\sin^2\theta_2\cos\theta_2+\sin^2\theta_1\cos\theta_1|}{\Gamma(x)}.\label{A7}
	 \end{align}
	 The first integral in \eqref{A7} has the bound \begin{align*}
	 	\lambda_N^2\left|\frac{k_1}{N}\right|\int \ud x\frac{\fR_{Nx,k_1-Nx}^N}{|N^{-1}k_1-x|\Gamma(x)}
	 	\lesssim\lambda_N^2\left|\frac{k_1}{N}\right|\int_0^1\frac{\ud r}{r^2+|N^{-1}k_1|^2}
	 	\lesssim \lambda_N^2.
	 \end{align*}
 For the second integral in \eqref{A7}, we have $$|\sin^2\theta_2\cos\theta_2+\sin^2\theta_1\cos\theta_1|\lesssim|\cos^2\theta_1-\cos^2\theta_2|+|\cos\theta_1+\cos\theta_2|,$$ and the last term $|\cos\theta_1+\cos\theta_2|$ equals
	 \begin{align*}
	 	\left|\frac{k_1\cdot(k_1/N-x)}{|k_1||k_1/N-x|}+\frac{k_1\cdot x}{|k_1||x|}\right|
	 	\leq&\frac{|k_1\cdot(k_1/N-x)+k_1\cdot x|}{|k_1||k_1/N-x|}+\frac{|k_1\cdot x|}{|k_1|}\left|\frac{1}{|k_1/N-x|}-\frac{1}{|x|}\right|\\
	 	\lesssim&\frac{|k_1/N|}{|k_1/N-x|}\,.
	 \end{align*}
	 Combining with \eqref{A8} and \eqref{A9}, \eqref{A3} holds for $i=1$.

\noindent {\bf Step 3} We now replace $\Gamma(x)$ with $\Gamma_1(x)$, i.e. define $P^N_3$ as
\begin{align}
	\lambda_N^2\int \ud x\frac{2\fR_{Nx,k_1-Nx}^N\sin^2\theta_1\cos^2\theta_1}{\Gamma_1(x)(1+G(L^N(N^2\Gamma_1(x))))}.
\end{align}
Recall that $G$ is continuous. Since $N^2\Gamma(x)\gtrsim1$, we have $L^N(N^2\Gamma(x))\in[0,C],$ for some $C$. 
{\eqref{A3} can be easily seen to hold by arguing as in the proof of \cite[(A.8)]{CET23b}.}
To apply a change of variables, we also define $P^N_4$ as
\begin{align}
	\lambda_N^2\int \ud x\frac{2\fR_{Nx,k_1-Nx}^N\sin^2\theta_1\cos^2\theta_1}{\Gamma_1(x)(1+\Gamma_1(x))(1+G(L^N(N^2\Gamma_1(x))))},
\end{align}
and \eqref{A3} is immediate for $i=3$. 

\noindent {\bf Step 4} At this step, we replace $P^N_4(k_{1:n})$ by
\begin{align}\label{PN5}
P^N_5(k_{1:n}):=\lambda_N^2\int_{|x|\leq1} \ud x\frac{2\sin^2\theta_1\cos^2\theta_1}{\Gamma_1(x)(1+\Gamma_1(x))(1+G(L^N(N^2\Gamma_1(x))))}.
\end{align}
Arguing as Step3 in \cite[Proposition A.1]{CGT24}, since $|k_1|\leq \frac N2$, and $\Gamma_1\gtrsim|x|^2$, $\sup_{k_{1:n}\in\ZZ_0^d}|P_4^N(k_{1:n})-P^N_{5}(k_{1:n})|$ is bounded by a constant times
	$$\lambda_N^2\int_{\frac12\leq|x|\leq1}\frac{1}{|x|^2}\ud x\lesssim\lambda_N^2.$$
At last, we focus on $P^N_5(k_{1:n})$. Set $\alpha_N:=|\frac{k_{1:n}}{N}|^2$. Then we have 
	\begin{align*}
		P^N_5(k_{1:n})=&\frac{\lambda^2}{\log N}\int_{|x|\leq1} \frac{2\sin^2\theta_1\cos^2\theta_1\ud x}{(2\pi)^2(2|x|^2+\alpha_N)((2\pi)^2(2|x|^2+\alpha_N)+1)(1+G(L^N(N^2(2\pi)^2(2|x|^2+\alpha_N))))}\\
		=&\frac{\lambda^2}{8\pi^2\log N}\int_0^{2\pi}\frac{\sin^22\theta}{2}\ud\theta\int_{0}^1 \frac{\ud r}{(2r+\alpha_N)((2\pi)^2(2r+\alpha_N)+1)(1+G(L^N(N^2(2\pi)^2(2r+\alpha_N))))}\\
		=&\frac{\lambda^2}{16\pi\log N}\int_{0}^1 \frac{\ud r}{(2r+\alpha_N)((2\pi)^2(2r+\alpha_N)+1)(1+G(L^N(N^2(2\pi)^2(2r+\alpha_N))))}\,.
	\end{align*}
	Let $t:=L^N(N^2(2\pi)^2(2r+\alpha_N))$. Then $\ud t=\frac{-2\lambda^2}{\log N(2r+\alpha_N)(1+(2\pi)^2(2r+\alpha_N))}\ud r$. The above integral equals
	\begin{align*}
		\frac{1}{32\pi}\int_{L^N(N^2(2\pi)^2(2+\alpha_N))}^{L^N(N^2(2\pi)^2\alpha_N)} \frac{\ud t}{1+G(t)}.
	\end{align*}
 It is immediate to verify that, with an error of order $\lambda_N^2$, we can replace the lower integration index with 0, and obtain
	\begin{align*}
	&\frac{1}{32\pi}\int_{0}^{L^N((2\pi)^2|k_{1:n}|^2)} \frac{\ud t}{1+G(t)}.
\end{align*}
Thus by the definition of $G$ in \eqref{eq:G}, the proof is complete.

\end{proof}

\renewcommand{\theequation}{C.\arabic{equation}}

\section{Proof of Theorem \ref{thm:F-D} when $d\geq3$}\label{appendixB}
{In this appendix, we prove Theorem \ref{thm:F-D} given Proposition \ref{prop:limite}. By Proposition \ref{prop:limite}, arguing as in the proof of \cite[Lemma 2.16, Lemma 2.17 and Proposition 2.13]{CGT24},
for any $a_1,a_2\in2\NN\,, m_1,m_2\geq2$, we have
\begin{align*}
	&\lim_{N\to\infty} \sum_{{p^1\in \Pi^{(n)}_{a_1,1},\, p^2\in \Pi^{(n)}_{a_2,1}}}\langle \mathcal T^N_{p^1} \sigma_{\bj,\bt},\overline {\mathcal T^N_{p^2} \sigma_{\bj,\bt}}\rangle= c(a_1,\lambda) c(a_2,\lambda)
	\|\sigma_{\bj,\bt}\|^2\,,\\
	&\lim_{N\to\infty}\sum_{{p^1\in \Pi^{(n)}_{a_1,m_1},\, p^2\in \Pi^{(n)}_{a_2,m_2}}} \langle(-\fL_0)^{-1}\fT_{p^1}^N\sigma_{\bj,\bt},\overline{\fT_{p^2}^N\sigma_{\bj,\bt}}\rangle= 0\,,
\end{align*}
where $c(a,\lambda)=1$ if $a=0$, and for $a\geq2$, $c(a,\lambda)$ is the one in Proposition \ref{prop:limite}.} 
Then similar as the proof of \cite[Proposition 2.12]{CGT24} by replacing $\{e_k\}_k$ with $\{\sigma_{k,\alpha}\}_{k,\alpha}$, we have the following Proposition.
\begin{prop}\label{prop:Nlim}
	For $d\geq3$ and $n$ fixed, there exists a unique constant $D^n>0$ such that, for $i=2,3$, 
	\begin{align*}
		\lim_{N\to\infty}\|(-\fL_0)^{-\frac12}(\fA_-^Nv_i^{N,n}-D^{n-i+2}\fL_0 f_{i-1})\|=0\,,\quad
		\lim_{N\to\infty}\|v^{N,n}\|=0\,.
	\end{align*}
\end{prop}
Now We prove Theorem \ref{thm:F-D} using Proposition \ref{prop:Nlim}. \eqref{5.2} follows from the proof of \cite[Theorem 2.7]{CGT24} and \eqref{eq:sector}. We only need to prove \eqref{5.3}.
Since the sequence $\{D^n\}_n$ is Cauchy by the proof of \cite[Theorem 2.11]{CGT24} in $d\geq3$, we only need to prove that the limit of $\{D^n\}_n$ is positive.

Fix $k\in\ZZ_0^d$.
Let $v^{N,n}[-k]$ be the solution to \eqref{eq:geneq} with $f_1=\sigma_{k,1}$ and $i=2.$
Then we have
\begin{align}
	D^n2\pi|k|=&\|(-\fL_0)^{-\frac12}(-D^n\fL_0)f_1\|\nonumber\\
	\leq&
\|(-\fL_0)^{-\frac12}\fA_-^Nv_2^{N,n}[-k]\|-\|(-\fL_0)^{-\frac12}(\fA_-^Nv_2^{N,n}[-k]-D^{n}\fL_0 f_1)\|.\label{C1}
\end{align}
By Proposition \ref{prop:Nlim}, the second term vanishes as $N\to\infty$. For the first term, by \eqref{eq:div}, we have
\begin{align*}
	\fA_-^Nv_2^{N,n}[-k](x)=-\sum_{\alpha=1}^{d-1}\sum_{j\in\ZZ_0^d}\langle v_2^{N,n}[-k],\fA_+^N\sigma_{-j,\alpha}\rangle \sigma_{j,\alpha}(x)\,.
\end{align*}
Recall that for $1\leq i\leq d$, $M_i$ is the momentum operator and commutes with $\fL_0,\fA^N_{+},\fA^N_{-}$ by \eqref{eq:commute}. Then $-\fL^N_{2,n}M_iv^{N,n}=\fA^N_+M_i\sigma_{k,1}=\iota2\pi k_i\fA^N_+\sigma_{k,1},$ and $M_iv^{N,n}=\iota2\pi k_iv^{N,n}.$ For $1\leq i\leq d$, $j\in \ZZ_0^d$,
\begin{align*}
	\langle v_2^{N,n}[-k],\fA_+^N\sigma_{-j,\alpha}\rangle&=
	\frac{1}{\iota2\pi k_i}\langle M_iv_2^{N,n}[-k],\fA_+^N\sigma_{-j,\alpha}\rangle
	=-\frac{1}{\iota2\pi k_i}\langle v_2^{N,n}[-k],\fA_+^NM_i\sigma_{-j,\alpha}\rangle\\
	&=\frac{j_i}{k_i}\langle v_2^{N,n}[-k],\fA_+^N\sigma_{-j,\alpha}\rangle\,.
\end{align*}
Therefore $\fA_-^Nv_2^{N,n}[-k]$ only has the $k-$th component, i.e.
\begin{align}\label{C.9}
	\fA_-^Nv_2^{N,n}[-k](x)=\sum_{\alpha=1}^{d-1}\langle \fA_-^Nv_2^{N,n}[-k],\sigma_{-k,\alpha}\rangle \sigma_{k,\alpha}(x)\,.
\end{align}
Since $\langle\fA_-^Nv_2^{N,n}[-k],\sigma_{-k,\alpha}\rangle=-\langle (-\fL^N_{2,n})^{-1}\fA^N_+\sigma_{k,1},\fA^N_+\sigma_{-k,\alpha}\rangle,$
by \cite[(5.66)]{CT24} and Fubini's Theorem,
\begin{align}
	\langle\fA_-^Nv_2^{N,n}[-k],\sigma_{-k,\alpha}\rangle=&-(2\pi)^2|k|^2\int_0^\infty e^{-s}\langle e^{sT^N_{2,n}}T^{N,+}\sigma_{k,1},T^{N,+}\sigma_{-k,\alpha}\rangle ds\nonumber\\
	=&-(2\pi)^2|k|^2\int_0^\infty e^{-s}\sum_{a=0}^\infty\frac{s^a}{a!}\langle (T^N_{2,n})^aT^{N,+}\sigma_{k,1},T^{N,+}\sigma_{-k,\alpha}\rangle ds.\label{C.3}
\end{align}
By \eqref{eq:liminner1} (still holds for each $N$ without taking the limit), 
for each $a\in\NN,$ $-\langle (T^N_{2,n})^aT^{N,+}\sigma_{k,1},T^{N,+}\sigma_{-k,\alpha}\rangle=\sum_{p\in\Pi_{a+2,1}^{(n)}}\langle \fT_p^N\sigma_{k,1},\sigma_{-k,\alpha}\rangle$ vanishes if $\alpha\neq1.$ Therefore we have
\begin{align}\label{C.10}
	\fA_-^Nv_2^{N,n}[-k](x) = \langle \fA_-^Nv_2^{N,n}[-k],\sigma_{-k,1}\rangle \sigma_{k,1}(x) = -\langle (-\fL^N_{2,n})^{-1}\fA^N_+\sigma_{k,1},\fA^N_+\sigma_{-k,1}\rangle \sigma_{k,1}(x),
\end{align}
which implies that
$\|(-\fL_0)^{-\frac12}\fA_-^Nv_2^{N,n}[-k]\|\geq\frac{1}{2\pi|k|}\|(-\fL_0)^\frac12v^{N,n}[-k]\|^2.$
By variational formula and \eqref{eq:sector}, there exists a constant $C$ such that
\begin{align*}
&\|(-\fL_0)^\frac12v^{N,n}[-k]\|^2\geq\frac{1}{1+C}\|(-\fL_0)^{-\frac12}\fA_+^N\sigma_{-k,1}\|^2\\
=&\frac{\lambda_N^2}{4(1+C)}\sum_{l_{1:2}}\sum_{l+m=k}\frac{|\fR_{l,m}^N|^2}{(|l|^2+|m|^2)}\left|(\hat\Lp(l)k)(l_1)(\hat\Lp(m)a_{k,1})(l_2)+(\hat\Lp(m)k)(l_2)(\hat\Lp(l)a_{k,1})(l_1)\right|^2.
\end{align*} 
Thus, as $N\to\infty,$ $D^n$ is greater than 
\begin{align*}
	&\frac{1}{(2\pi)^2|k|^2}\frac{\lambda^2}{1+C}\int_{|x|\leq1}\frac{1}{4|x|^2}\left[|k|^2\sin^2\theta_1\left(1-\frac{(a_{k,1}\cdot x)^2}{|x|^2}\right)+\frac{(k\cdot x)^2(a_{k,1}\cdot x)^2}{|x|^4}\right]\ud x\\
	\geq&\frac{1}{(2\pi)^2|k|^2}\frac{\lambda^2}{1+C}\int_{|x|\leq1}\frac{|k|^2\cos^2\theta_1\cos^2\theta_2}{4|x|^2}\ud x,
\end{align*}
where $\theta_1$ ($\theta_2$) is the angle between $x$ and $k$ ($a_{k,1}$).
Thus the limit $D$ of $D^n$ is strictly positive.

\textbf{Acknowledgements:} 
The authors would like to thank Huanyu Yang and Lukas Gr\"{a}fner for helpful discussions. 

\def\cprime{$'$} \def\ocirc#1{\ifmmode\setbox0=\hbox{$#1$}\dimen0=\ht0
	\advance\dimen0 by1pt\rlap{\hbox to\wd0{\hss\raise\dimen0
			\hbox{\hskip.2em$\scriptscriptstyle\circ$}\hss}}#1\else {\accent"17 #1}\fi}


\begin{thebibliography}{BDLSVMMM88}
	\bibitem[BGME22]{BGME22}  D. Bandak, N. Goldenfeld, A. A. Mailybaev, G. Eyink.
	Dissipation-range fluid turbulence and thermal noise. {\em Physical Review E}, 105(6), 065113, 2022.
	
	\bibitem[CES21]{CES21} G. Cannizzaro, D. Erhard, P. Schönbauer. 2D anisotropic KPZ at stationarity: Scaling, tightness and nontriviality. {\em  Ann. Probab.}, 49(1), 122-156, 2021.
	
	\bibitem[CET23a]{CET23a} G. Cannizzaro, D. Erhard, F. Toninelli. The stationary AKPZ equation: Logarithmic superdiffusivity. {\em Commun. Pure Appl. Math.}, 76(11), 3044–3103, 2023.
	
	\bibitem[CET23b]{CET23b} G. Cannizzaro, D. Erhard, F. Toninelli. Weak coupling limit of the Anisotropic KPZ equation. {\em Duke Math.J.}, 172(16), 3013–3104, 2023.
	
	\bibitem[CG24]{CG24} G.Cannizzaro, H. Giles. An invariance principle for the 2d weakly self-repelling Brownian polymer. {\em Probab. Theory Relat. Fields}, 2025.
	
	\bibitem[CGT24]{CGT24} G. Cannizzaro, M. Gubinelli, F. Toninelli. Gaussian Fluctuations for the Stochastic Burgers
	Equation in Dimension $d\geq2$. {\em  Commun. Math. Phys.}, 405, 89 (2024).
	
	\bibitem[CMT24]{CMT24} G. Cannizzaro, Q. Moulard, F. Toninelli. Superdiffusive central
	limit theorem for the stochastic burgers equation at the critical dimension. {\em arXiv:2501.00344}, 2024.
	
	\bibitem[CT24]{CT24} G. Cannizzaro, F. Toninelli. Lecture notes on stationary critical and super-critical SPDEs. {\em arXiv:2403.15006}, 2024.
	
	\bibitem[Duc22]{Duc22} P. Duch. Renormalization of singular elliptic stochastic PDEs using flow equation. {\em Probability and Mathematical Physics}, 6(1), 111-138, 2025.
	
	\bibitem[GHW23]{GHW23} B. Gess, D. Heydecker, Z. Wu. Landau-Lifshitz Navier-Stokes equations: Large deviations and relationship to the energy equality. {\em arXiv:2311.02223}, 2023.
	
	\bibitem[GIP15]{GIP15} M. Gubinelli, P. Imkeller, N. Perkowski. Paracontrolled distributions and singular PDEs. {\em Forum Math. Pi}, vol.3, 2015.
	
	\bibitem[GJ13]{GJ13} M. Gubinelli, M. Jara. Regularization by noise and stochastic Burgers equations. {\em Stoch. Partial Differ. Equ. Anal.} Comput. 1, no. 2, (2013), 325–350.
	
	\bibitem[GPP24]{GPP24} L. Gr\"afner, N. Perkowski, S. Popat. Energy solutions of singular SPDEs on Hilbert
	spaces with applications to domains with boundary conditions. {\em arXiv:2411.07680}, 2024.
	
	\bibitem[Gra25]{Gra25} L. Gr\"afner. Critical and supercritical singular s(p)des: A probabilistic approach, Dissertation, 2025.
	
	\bibitem[Hai14]{Hai14} M. Hairer. A theory of regularity structures. {\em Invent.Math.}, 198, 269-504, 2014.

	\bibitem[IK02]{IK02}  K. Ito and F.Kappel. Evolution equations and approximations, vol 61. of {\em Ser. Adv. Math. Appl. Sci.} Singapore: World Scientific, 2002.
	
	\bibitem[JP24]{JP24}
	R. Jin, N. Perkowski. Fractional stochastic Landau-Lifshitz Navier-Stokes equations in dimension $d\geq3$: Existence and (non-)triviality. {\em arXiv:2403.04911 }, 2024.
	
	\bibitem[Kal21]{Kal21} O. Kallenberg. Foundations of Modern Probability, vol. 99 of Probability
	Theory and Stochastic Modelling. Springer International Publishing, Cham, 2021.
	
	\bibitem[KLO12]{KLO12} T. Komorowski, C. Landim, S. Olla. Fluctuations in Markov Processes: Time Symmetry and Martingale Approximation. Springer-Verlag Berlin Heidelberg 2012.
	
	\bibitem[LL87]{LL87} L. D. Landau, E. M. Lifshitz. {\em Course of theoretical physics. Vol. 6.} Pergamon Press, Oxford, second edition, 1987. Fluid mechanics, Translated from the third Russian edition by J. B. Sykes and
	W. H. Reid. 
	
	\bibitem[Mit83]{Mit83} I. Mitoma. Tightness of probabilities on $C([0,1];\fS′)$ and $D([0,1];\fS′)$. {\em Ann.
	Probab.} 11, no. 4, (1983), 989–999.
	
	\bibitem[MW17]{MW17} J.-C. Mourrat, H. Weber. Convergence of the Two-Dimensional Dynamic Ising-Kac Model to $\Phi^4_2$. Communications on Pure and Applied Mathematics, 70 (4), 2016, 717–812.
	
	\bibitem[Nua06]{Nua06} D. Nualart. 
	The Malliavin Calculus and Related Topics. Probability and its Applications. Springer Berlin, Heidelberg, 2006.
	
	\bibitem[QY98]{QY98} J. Quastel, H-T Yau. Lattice gases, large deviations, and the incompressible Navier
	Stokes equations. {\em Ann. of Math.}, 148(1), 51–108, 1998.
\end{thebibliography}
\end{document}